\documentclass[12pt]{amsart}

\usepackage[hmargin=0.8in,height=8.6in]{geometry}
\usepackage{amssymb,amsthm, times}
\usepackage{delarray,verbatim}
\usepackage{ifpdf}
\ifpdf
\usepackage[pdftex]{graphicx}
 \else
\usepackage[dvips]{graphicx}

 \fi

\usepackage{bm}

\usepackage{ifpdf}
\usepackage{color}
\definecolor{webgreen}{rgb}{0,.5,0}
\definecolor{webbrown}{rgb}{.8,0,0}
\definecolor{emphcolor}{rgb}{0.95,0.95,0.95}

\usepackage{hyperref}
\hypersetup{%
          colorlinks=true,
          linkcolor=webbrown,
          filecolor=webbrown,
          citecolor=webgreen,
          breaklinks=true}
\ifpdf \hypersetup{pdftex,
            pdfstartview=FitH, %%Fit, FitB, FitH
            bookmarksopen=true,
            bookmarksnumbered=true
} \else \hypersetup{dvips} \fi

\newtheorem{theorem}{Theorem}[section]

\newtheorem{proposition}{Proposition}[section]
\newtheorem{corollary}{Corollary}[section]
\newtheorem{definition}{Definition}[section]
\newtheorem{example}{Example}[section]
\newtheorem{remark}{Remark}[section]

\title{Worst portfolios for dynamic monetary utility processes}
\thanks{This version: \today. }
\author[D.\ Hern\'andez-Hern\'andez]{Daniel Hern\'andez-Hern\'andez}
\address[D.\ Hern\'andez-Hern\'andez]{Department of Probability and Statistics,
Centro de Investigaci\'on en Matem\'aticas, Apartado Postal 402, Guanajuato GTO 36000, Mexico.}
\email{dher@cimat.mx}
\author[O.H. Madrid-Padilla]{Oscar Hernan Madrid-Padilla}
\address[O. H. Madrid-Padilla]{Department of Statistics and Data Sciences, The University of Texas, Austin}
\email{oscar.madrid@utexas.edu }
\date{}

\begin{document}

\maketitle

\begin{abstract}
We study the worst portfolios
for a class of law invariant dynamic monetary utility functions with
domain in a class of stochastic processes.  The concept
of comonotonicity is introduced for these processes in order  to prove the existence of
worst portfolios. Using robust representations   of   monetary
utility function processes in discrete time,   a relation
between the worst portfolios at different periods of time is presented. Finally, we study conditions
to achieve the maximum  in the representation theorems for concave monetary utility functions
that are continuous for bounded decreasing sequences. 
\end{abstract}

{\noindent \small{\textbf{Keywords:}\,\text{Comonotonicity, Best portfolio, Monetary utility function
process, Time-consistency, Relevancy}

\noindent \small{\textbf{Mathematics Subject Classification (2010):}\, 91B30, 91B16

\section{ Introduction}

%\begin{spacing}{0.69999999999999996}

In this paper, we present a definition  of  worst case portfolios for insurance versions of dynamic monetary utility functions, borrowing  the notion of insurance version from Ruschendorf \cite{ruschendorf2012worst}. This definition extends the ideas of worst case portfolios presented in \cite{ruschendorf2012worst}  to the dynamic  framework of  monetary utility functions studied in \cite{cheridito2006dynamic}.  Within the discrete time  framework, given an agent  interested in measuring the maximum risk,  we study conditions under which a portfolio preserves the property of worst case portfolio along the time.

  Monetary utility functions can be built through risk measures, and vice versa. When the aim is to look
at   worst case scenarios,   risk measures are used to quantify the risk, while  when the best
possible performances are desired   utility functions are employed.
The axiomatic notion of monetary risk measure was first introduced
by Artzner et. al.  \cite{artzner1999coherent,artzner1997a1} , and has been extensively
studied since then in different directions. The general framework
defined by Follmer and Schied \cite{follmer2011stochastic} in terms of convex functionals
has also been analyzed in a dynamic setting, providing a systematic
axiomatic approach to time-consistent monetary risk measures; see,
for instance, Cheridito et. al. \cite{cheridito2006dynamic}, Cheridito and Kupper \cite{cheridito2011composition},
and the recent work of Vioglio et. al. \cite{cerreia2011risk}. The research work
developed on this area has led to the general definition of conditional
monetary risk measure \cite{cheridito2006dynamic}, where the concept of monetary
utility process is also defined. Under a different perspective, conditional
monetary functions have also been studied by Filipovic et.al \cite{filipovic2012approaches}.
On the other hand, from the regulators point of view, they are only
interested in the amount and the intensity of the risk, and not in
its operational nature, which has motivated the study of law invariant
risk measures \cite{danielsson2001academic}, \cite{ekeland2012comonotonic}, \cite{ruschendorf2006law}. The analysis of these
monetary risk measures that are continuous from below as well as their
representation can be found in \cite{follmer2011stochastic}, for one period of time.

The problem of identifying the worst case dependence structure of
a d-dimensional portfolio has been analyzed by Ruschendorf \cite{ruschendorf2012worst}.
In this paper we relate this concept with conditional monetary risk
measures. Understanding the structure of worst case portfolios associated
to an specific monetary risk measure is important in order to calculate the
aggregated risk of a given portfolio \cite{burgert2006consistent},  \cite{ruschendorf2002n}.
As it is explained by Embrechts et.al.  \cite{embrechts2013model}, typically only the marginal distributions
are available and new techniques have to be developed to calculate
the most conservative values of the associated risk; for an interesting
connection of this problem with the mass transportation problem we refer the reader
to \cite{rachev1998mass}.  

In this work we are interested in measuring the average risk of an n-dimensional vector of financial positions evolving in
	time; each financial position is modeled as a stochastic process.
	One of our goals is to study properties of portfolios that are worst in the sense of having maximum risk. From a practical  point of view these type of portfolios are important, since allow us to quantify the riskiest situation for aggregated positions. Our aim is to study optimal portfolios in aggregated sense, with respect to fixed marginals of the different stochastic process involved.  More precisely, the initial aim of this paper is to describe the worst
case portfolios of insurance versions corresponding to  law invariant conditional monetary utility functions
that are continuous for bounded decreasing sequences. In that direction,
we present a new definition of comonotonicity, which is fundamental
in the proof of the main results presented in this paper; this concept was also essential in the
work of Ekeland et. al. \cite{ekeland2012comonotonic}. The first result that we present, Theorem \ref{Thm3.1}, establishes that in order
to study worst case portfolios, we should
understand the comonotones portfolios associated with the worst
scenarios of the average risk function. This is a generalization of
Theorem 3.2 in \cite{ruschendorf2012worst}. Then, we make a transition to  study 
the worst case portfolios of insurance versions of discrete-time conditional monetary utility
function processes. Before proving our main result an invariance
property of being worst case portfolio in three-time steps is established,
for certain class of insurance versions associated with time-consistent monetary utility function processes.
The main result of this work, stated as  Theorem \ref{MainThm}, provides conditions on the insurance version of a monetary
utility function process in order to  guarantee that the property of being
a worst case portfolio does not change over time. In the final part
of this paper, we are interested in  establishing conditions for
reaching the maximum in  the dual Fenchel representation formulas for the insurance version of the monetary utility function; see, for instance, 
Cheridito et.al. \cite{cheridito2006dynamic}. This is used  latter within our study of worst portfolios.

The structure of the paper is the following. Section 2 is dedicated
to introduce the notation and spaces needed throughout the paper,
while in Section 3 we begin defining a new concept of comonotonicity
which is based on conditional expectation. Moreover, we provide a study
of some aspects of worst case portfolios of insurance versions of monetary utility function
processes. Although Theorem \ref{Thm3.2} is interesting by itself, the main
result of that section, and of this paper, is Theorem \ref{MainThm}. Finally,
some results related with dual representations are presented in Section
4, providing conditions to attain the maximum in such representations.

\section{ Preliminaries}

Throughout the remainder $\left(\Omega,\mathcal{F},\left(\mathcal{F}_{t}\right)_{t\in\mathbb{N}},P\right)$
is a filtered probability space, with $\mathcal{F}_{0}$ $=$ $\left\{ \emptyset,\Omega\right\} $.
All equalities and inequalities between random variables or stochastic
processes are understood in the $P-$almost sure sense even without
explicit mention. Also, $T$ is a fixed deterministic finite time
horizon in $\mathbb{N}$. Denoting by $E$ the expectation operator
with respect to $P$ we introduce the following spaces and operators,
which are important to give the precise definitions.

$\;$
%\end{spacing}
\begin{itemize}
\item The space of $\{\mathcal{F}_{t}\}-$ adapted stochastic processes
is denoted by $\mathcal{R}^{0},$ and define the subclass

\[
\mathcal{R}^{\infty}:=\left\{ X\in\mathcal{R}^{0}:\,\parallel X\parallel_{\mathcal{R}^{\infty}}<\infty\right\} ,
\]

with $\parallel X\parallel_{\mathcal{R}^{\infty}}:=\inf\left\{ m\in\mathbb{R}:\sup_{t\in\mathbb{N}}\mid X_{t}\mid\,\leq m\right\} $.

\item $A^{1}:=\left\{ a\in\mathcal{R}^{0}:\,\parallel a\parallel_{A^{1}}\,<\infty\right\} $,
where $a_{-1}=0,\quad\triangle a_{t}:=a_{t}-a_{t-1},\; t\in\mathbb{N},\quad \text{and}\quad\parallel a\parallel_{A^{1}}:=E\left(\sum_{t\in\mathbb{N}}\mid\triangle a_{t}\mid\right)$.
\item $A_{+}^{1}:=\left\{ a\in A^{1}\mid\triangle a_{t}\geq0\; \text{for\; all}\; t\in\mathbb{N}\right\} $,
and the bilineal form $\prec.,.\succ$ defined on $\mathcal{R}^{\infty}$$\times$
$A^{1}$ is given by

\textsf{
\[
\prec X,a\succ\,:=E\left(\sum_{t\in\mathbb{N}}X_{t}\triangle a_{t}\right).
\]
}

\item The space $\mathcal{R}^{\infty}$ is endowed with the weak
topology $\sigma$$\left(\mathcal{R}^{\infty},A^{1}\right)$, that
makes all the linear functionals $X \to  \langle X, a\rangle,\; a\in A^1$, continuous,  and
analogously, $\sigma\left(A^{1},\mathcal{R}^{\infty}\right)$ denotes
the weak topology on $A^{1}$. 
\item Given the $\{\mathcal{F}_{t}\}-$ stopping times $\tau$ and $\theta$
such that $0$ $\leq$ $\tau$ $<$ $\infty$ and $\tau$ $\leq$
$\theta$ $\leq$ $\infty$, the projection $\pi_{\tau,\theta}$ $:$
$\mathcal{R}^{0}$ $\longrightarrow$ $\mathcal{R}^{0}$ is given
by $\pi_{\tau,\theta}\left(X\right)_{t}:=1_{\left\{ \tau\leq t\right\} }X_{t\wedge\theta},\, t\in\mathbb{N}$.
Define the vector space $\mathcal{R}_{\tau,\theta}^{\infty}:=\pi_{\tau,\theta}\left(\mathcal{R}^{\infty}\right)$.
\item For $X$ $\in$ $\mathcal{R}^{\infty}$ and $a$ $\in$ $A^{1},$
let

\[
\prec X,a\succ_{\tau,\theta}\,:=E\left(\sum_{t\in\left[\tau,\theta\right]\cap\mathbb{N}}X_{t}\triangle a_{t}\mid\mathcal{F}_{\tau}\right).
\]

\item Define the following subsets of $A^{1}:$

\[
A_{\tau,\theta}^{1}:=\pi_{\tau,\theta}A^{1},\quad\left(A_{\tau,\theta}^{1}\right)_{+}:=\pi_{\tau,\theta}A_{+}^{1},\quad D_{\tau,\theta}:=\left\{ a\in\left(A_{\tau,\theta}^{1}\right)_{+}\mid\prec1,a\succ_{\tau,\theta}\,=1\right\} .
\]

Finally, let $D_{\tau,\theta}^{e}:=\left\{ a\in D_{\tau,\theta}\mid P\left(\sum_{j\geq t\wedge\theta}\bigtriangleup a_{j}>0\right)=1\: \text{for\: all}\: t\in\mathbb{N}\right\} $.

\end{itemize}

Now let us recall the definition of a monetary utility function in
the static framework.

\begin{definition}
A mapping $\phi$ : $\mathcal{R}_{\tau,\theta}^{\infty}$ $\longrightarrow$
$L^{\infty}\left(\mathcal{F}_{\tau}\right)$ is a monetary utility
function on $\mathcal{R}_{\tau,\theta}^{\infty}$ if the following
three properties hold:\\
\\
$\left(0\right)$ $\phi\left(1_{A}X\right)$
$=$ $1_{A}\phi\left(X\right)$ for all $X$ $\in$ $\mathcal{R}_{\tau,\theta}^{\infty}$
and $A$ $\in$ $\mathcal{F}_{\tau}$.\\
\\
$\left(1\right)$ $\phi\left(X\right)$ $\leq$
$\phi\left(Y\right)$ for all $X$, $Y$ $\in$ $\mathcal{R}_{\tau,\theta}^{\infty}$
such that $X$ $\leq$ $Y.$\\
\\
$\left(2\right)$ $\phi\left(X+m1_{[\tau,\infty)}\right)$
$=$ $\phi\left(X\right)$ $+$ $m$ for all $X$ $\in$ $\mathcal{R}_{\tau,\theta}^{\infty}$
and $m$ $\in$ $L^{\infty}\left(\mathcal{F}_{\tau}\right).$
\\
\\
Such a mapping is said to be:\\
\\
$\left(3\right)$ Concave
if $\phi\left(\lambda X+\left(1-\lambda\right)Y\right)$ $\geq$ $\lambda\phi\left(X\right)$
$+$ $\left(1-\lambda\right)\phi\left(Y\right)$ for all $X,$ $Y$
$\in$ $\mathcal{R}_{\tau,\theta}^{\infty}$ and $\lambda$ $\in$
$L^{\infty}\left(\mathcal{F}_{\tau}\right)$ such that $0$ $\leq$
$\lambda$ $\leq$ $1$.\\
\\
 $\left(4\right)$  Coherent
if $\phi\left(\lambda X\right)$ $=$ $\lambda\phi\left(X\right)$
for all $X$ $\in$ $\mathcal{R}_{\tau,\theta}^{\infty}$ and $\lambda$
$\in$ $L_{+}^{\infty}\left(\mathcal{F}_{\tau}\right)$.\\
\\
$\left(5\right)$ Continuous
for bounded decreasing sequences if $\lim_{n\rightarrow\infty}$$\phi\left(X^{n}\right)$
$=$ $\phi\left(X\right)$  for every decreasing sequence $\left\{ X^{n}\right\} _{n\in\mathbb{N}}$
in $\mathcal{R}_{\tau,\theta}^{\infty}$ and $X$ $\in$ $\mathcal{R}_{\tau,\theta}^{\infty}$,
such that $X_{t}^{n}$ $\rightarrow$ $X_{t}$  for all $t$ $\in$
$\mathbb{N}$.\\
\\
$\left(6\right)$ $\theta-$ relevant if $A$ $\subset$ $\left\{ \phi\left(-\varepsilon1_{A}1_{[t\wedge\theta,\infty)}\right)<0\right\} $
for all $\varepsilon$ $>$ $0,$ $t$ $\in$ $\mathbb{N}$ and $A$
$\in$ $\mathcal{F}_{t\wedge\theta}.$
\end{definition}

The acceptance set of a monetary utility function $\phi$ is defined
as \begin{equation}\label{aceptance-set}\mathcal{C_{\phi}}=\left\{ X\in\mathcal{R}_{\tau,\theta}^{\infty}\mid\phi\left(X\right)\geq0\right\}. \end{equation}
\noindent
The notion of relevance introduced by Follmer and Schied \cite{follmer2011stochastic} captures the intuitive fact that if $X$ is a non--positive random variable
with positive probability of being negative, then the risk of $X$
is higher than that of the position identically zero. The analogous
version of this concept for $\mathcal{F_{\tau}}$ conditional monetary
utility functions was given by Cheridito et.al. \cite{cheridito2006dynamic}.%{[}3, Definition 3.19{]}.

\begin{definition}
A mapping $\gamma$ from $D_{t,T}$ to the space of measurable functions
$f$ $:$ $\left(\Omega,\mathcal{F}_{t}\right)$ $\longrightarrow$
$\left[-\infty,0\right]$ is said to be a penalty function if
\[
{\rm ess}\; \sup_{a\in D_{t,T}}\gamma\left(a\right)=0.
\]
Such a function is called local if $\gamma\left(1_{A}a+1_{A^{c}}b\right)=1_{A}\gamma\left(a\right)+1_{A^{c}}\gamma\left(b\right)$
for all $a,$ $b$ $\in$ $D_{t,T}$ and $A$ $\in$ $\mathcal{F}_{t}.$
\end{definition}

\begin{remark}\label{Rem2.1}
  If $\phi$ is a monetary utility function, we naturally define the insurance version by $\Psi(\cdot)=-\phi(-(\cdot))$. It turns out that some times it is more convenient to work with the insurance version, and we often work with such functions instead of monetary utility functions. On the other hand, considering the negative of a monetary utility function, the result
is a mapping that generalizes the original definition of a monetary
risk measure \cite{follmer2011stochastic}, namely,  the negative of a monetary utility function $\rho(\cdot)=-\phi(\cdot)$ defines a monetary risk measure.  
\end{remark}
\begin{definition}
(a) Given $S$ $\in$ $\mathbb{N}$,
with $S$ $\leq$ $T$, $\left(\phi_{t,T}\right)_{t\in\left[S,T\right]\cap\mathbb{N}}$
is a monetary utility process if for each $t$
$\in$ $\left[S,T\right]$ $\cap$ $\mathbb{N}$, $\phi_{t,T}$
is a monetary utility function on $\mathcal{R}_{t,T}^{\infty}.$
When the properties of concavity, coherence, decreasing monotonicity
or relevancy are satisfied for  $\phi_{t,T}$  for each $t$
$\in$ $\left[S,T\right]$ $\cap$ $\mathbb{N}$, we say that the utility function process $\left(\phi_{t,T}\right)_{t\in\left[S,T\right]\cap\mathbb{N}}$ is concave, coherent, monotonically decreasing or relevant, respectively. If  $\tau$ is an $\left(\mathcal{F}_{t}\right)-$stopping
time, with $S\leq \tau\leq T$, we
define the mapping $\phi_{\tau,T}$$:$ $\mathcal{R}_{\tau,T}^{\infty}$
$\rightarrow$ $L^{\infty}\left(\mathcal{F_{\tau}}\right)$ by
%[OSCAR: EN EL LADO DERECHO NO APARECE THETA, ME PARECE QUE DEBERIA SER $\phi_{\tau,T}$].
\[
\phi_{\tau,T}\left(X\right):=\sum_{t\in\left[S,T\right]\cap\mathbb{N}}\phi_{t,T}\left(1_{\left\{ \tau=t\right\} }X\right).
\]
(b) Such a utility function process $\left(\phi_{t,T}\right)_{t\in\left[S,T\right]\cap\mathbb{N}}$
is time-consistent if
\[
\phi_{t,T}\left(X\right)=\phi_{t,T}\left(X1_{[t,\theta)}+\phi_{\theta,T}\left(X\right)1_{[\theta,\infty)}\right)
\]
for each $t$ $\in$$\left[S,T\right]$$\cap$ $\mathbb{N},$
every finite $\left(\mathcal{F}_{t}\right)-$stopping time $\theta$
such that $t$ $\leq$ $\theta$ $\leq$ $T$ and all process $X$$\in$
$\mathcal{R}_{t,T}^{\infty}.$
\end{definition}

\begin{example}
Given $\alpha>0$, let $u$ be the exponential
utility function

\[
u(x)=1-\exp\left(-\alpha x\right),\; x\in\mathbb{R}.
\]
The certainty equivalent of a probability measures of $\mathbb{R}$  or ``lottery'' $\mu$, is defined as the number $c\left(\mu\right)$ for which the identity

\[
u\left(c\left(\mu\right)\right)=U\left(\mu\right):=\int u(y)\mu(dy),
\]
is satisfied.

Let $t\in\left[0,T\right]\cap\mathbb{N}$ and define $U_{t}$ $:$
$L^{\infty}\left(\mathcal{F}_{T}\right)$ $\longrightarrow$ $L^{\infty}\left(\mathcal{F}_{t}\right)$
as

\[
U_{t}(Y)=E\left(u\left(Y\right)\mid\mathcal{F}_{t}\right).
\]
The function $C_{t}\left(Y\right)$$\in$ $L^{\infty}\left(\mathcal{F}_{t}\right)$
is named the certainty equivalent of $Y$ at time $t$ if

\[
u\left(C_{t}\left(Y\right)\right)=U_{t}\left(Y\right).
\]
It can be verified that

\[
C_{t}\left(Y\right)=-\frac{1}{\alpha}\log\, E\left[\exp\left\{ -\alpha Y\right\} \mid\mathcal{F}_{t}\right].
\]
Moreover, the utility function process $\left(\phi_{t,T}\right)_{t\in\left[0,T\right]\cap\mathbb{N}},$
with $\phi_{t,T}(X)$ $:=$ $C_{t}\left(X_{T}\right),$ $X$ $\in$
$\mathcal{R}_{t,T}^{\infty},$ is time-consistent \cite{cheridito2006dynamic}.
\end{example}

\section{Worst portfolios of conditional monetary utility
functions}\label{Section3}

In this section our aim is to analyze the relationship between the
notions of worst portfolio and conditional utility functions, extending to the dynamic case the results  established by Ruschendorf \cite{ruschendorf2012worst} for the static case between risk measures and worst portfolios. Given a risk
measure, the study of the joint distribution of vector-valued portfolios
evolving in time is crucial within the theory of Financial Mathematics,
while in practice, agents are enforced by regulators to quantify the
underlying risk associated with their positions, in order to take
appropriate decisions. Also, from a different perspective, risk measures
are useful to understand the endogenous effect of liquidity risk in economical
crisis; see, for instance, the discussion presented by Danielson et.al. \cite{danielsson2001academic}. Taking decisions using risk measures derives in some cases
in the analysis of worst portfolios \cite{ruschendorf2006law,ruschendorf2012worst}; in order
to describe worst d-dimensional portfolios we
employ max correlation of conditional monetary utility functionals, using conditional expectation. Recall that there is a bijective relation
between utility functions and risk measures (see Remark \ref{Rem2.1}),  and we shall formulate our results  in terms of worst portfolios, which are obtained by  quantifying the highest risk that  the investor  is facing; however, this quantification is restricted to a pre-specified set of marginal risks that  are averaged  to get  the overall risk.

Let $X$ $\in$ $\mathcal{R}_{t,T}^{\infty}$, and express it as a
random vector by $\left(X_{t},...,X_{T}\right)$; we write $X$ $\sim$
$X^{\prime}$ for $X^{\prime}$ $\in$ $\mathcal{R}_{t,T}^{\infty}$,
when $X$ and $X^{\prime}$ have the same distribution as random vectors.
%\textcolor{red}{Now, let $\phi_{t,T}$ be a monetary utility function, and define $\Psi_{t,T}\left(\cdot\right) := -\phi_{t,T}\left(-\left(\cdot\right)\right)$. This function extends the notion of insurance version as introduced in \cite{ruschendorf2012worst} to the dynamic setting described in this paper. Note that  it is not difficult to verify that  an insurance version is actually a dynamic monetary utility function. However, we will make explicit that distinction and often work with insurance versions given that arising from well studied dynamic monetary utility functions, insurance versions can be more amenable for studying worst portfolios. This will become clear in the next section. }
Moreover,  we also define  the function $\phi_{t,T}^{\#}$ on $A_{t,T}^{1}$ given as

\[
\phi_{t,T}^{\#}\left(a\right):={\rm ess}\, \inf_{X\in\mathcal{C}_{\phi_{t,T}}}\prec X,a\succ_{t,T},
\]
with $\mathcal{C}_{\phi_{t,T}}$ defined as in (\ref{aceptance-set}). 
%This set  plays an important role in proving the main theorems of this paper.

% function is in some sense the Legendre-Fenchel conjugate of $\phi_{t,T}\left(\cdot\right)$, since by Remark 3.13 in \cite{cheridito2006dynamic} %{[}3, Remark 3.13{]}

%\[
%\Psi_{t,T}^{\#}\left(a\right)={\rm ess}\, \inf_{X\in\mathcal{R}_{t,T}^{\infty}}\left\{ \prec X,a\succ_{t,T}-\Psi_{t,T}\left(X\right)\right\} ,\quad \text{for\; all}\quad a\in D_{t,T}.
%\]

We now introduce the definition of comonotonicity for the dynamic case  in an analogous  manner to  \cite{ruschendorf2012worst} for the static case. Starting with a  finite set of stochastic processes $X^{1},\ldots,X^{n}\in\mathcal{R}_{t,T}^{\infty}$  which  represent the evolution in time  of financial values, the approach presented below aims to quantify the worst case performance  when   conditional maximal correlations between these financial positions evolving in time and conditional densities in 
the space $\Omega \times \mathbb{N}$ are computed. Throughout, $X^{1},\ldots,X^{n}$ remain fixed.
%As in \cite{ruschendorf2012worst}, we 
Below we shall show how the problem of calculating worst portfolios in terms of risk measures can be analyzed using comonotonicity.
%The definitions of comonotonicity and the average risk function are presented below; they are motivated by the corresponding static versions studied in  \cite{ruschendorf2012worst}. Let $X^{1},\ldots,X^{n}\in\mathcal{R}_{t,T}^{\infty}$ be stochastic  processes representing  portfolios evolving in time, which shall remain fixed in the rest of this section.

\begin{definition}\label{Def3.1}
$\left(i\right)$ Let
$a$ $\in$ $D_{t,T}$ be fixed, and define the mapping $\Psi_{a}$
$:$ $\mathcal{R}_{t,T}^{\infty}$ $\longrightarrow$ $L^{\infty}\left(\mathcal{F}_{t}\right)$
as
\[
\Psi_{a}\left(\hat{X}\right)={\rm ess}\, \sup_{\tilde{X}\sim\hat{X}}\prec\tilde{X},a\succ_{t,T},\qquad\hat{X}\in\mathcal{R}_{t,T}^{\infty}.
\]

$\left(ii\right)$ The average risk function $F_{t,T}$ $:$ $D_{t,T}$$\longrightarrow$ $L^{\infty}\left(\mathcal{F}_{t}\right)$ is defined by

\[
F_{t,T}\left(a\right)\,=\,\frac{1}{n}\sum_{i=1}^{n}\Psi_{a}\left(X^{i}\right)+\phi_{t,T}^{\#}\left(a\right).
\]

$\left(iii\right)$ We call $a^{0}$$\in$ $D_{t,T}$ the worst scenario of $F_{t,T},$ if
\[
F_{t,T}\left(a^{0}\right)\,=\, {\rm ess}\, \sup_{a\in D_{t,T}}F_{t,T}\left(a\right).
\]

$\left(iv\right)$ If $a^{0}$$\in$ $D_{t,T}$ and
$\tilde{X}^{i}$ $\sim$ $X^{i},$ $\tilde{X}^{i}$ $\in$ $\mathcal{R}_{t,T}^{\infty}$,
$i$ $=$$1$,....,$n$, we call $\tilde{X}^{1}$,....,$\tilde{X}^{n}$
$a^{0}-comonotone$ when
\[
\Psi_{a^{0}}\left(\tilde{X}^{i}\right)\,=\,\prec\tilde{X}^{i},a^{0}\succ_{t,T},
\]
and

\[
\Psi_{a^{0}}\left(\sum_{i=1}^{n}\tilde{X}^{i}\right)\,=\,\prec\sum_{i=1}^{n}\tilde{X}^{i},a^{0}\succ_{t,T}.
\]
\end{definition}

%The function $F_{t,T}$$\left(a\right)$ is interpreted as the average
%risk of $X^{1},$....,$X^{n}$ with respect to $\Psi_{a}$ plus the
%``penalty'' $\phi_{t,T}^{\#}\left(a\right),$ depending on the density
%$a$ $\in$ $D_{t,T}$. Here, $\Psi_{a}$ generalizes the notion of
%correlation risk measure.

%We note that Definition 3.1 considers the functions $\Psi_a$ which are by themselves insurance versions. 

Note that each function $\Psi_a$ can be interpreted as the maximum correlation between a financial position, constrained to have fixed marginals, and a conditional density $a$ in $\Omega \times \mathbb{ N}$. Moreover, Part (ii) in Definition 3.1 constructs an average between a maximum correlation function  $\Psi_a$ and different positions plus a penalty on the density $a$.   On the other hand, maximizing $F_{t,T}$ over the set of densities $D_{t,T}$ resembles the problem of finding  a density with maximum correlation with the average of the financial positions. The following remark illustrates the deep connection between the  definition of comonotonicity and densities of maximal correlation, with given financial positions.

\begin{remark}\label{Rem3.1}
Let us illustrate the notion
of comonotonicity with a simple example when $t$ $=$ $1$ and $T$
$=$ $2.$ Given $a$ $\in$ $D_{1,2}$, with $\bar{a}$ $=$ $\left(\triangle a_{1},\triangle a_{2}\right)$,
choose $X^{0}$ $\in$ $\mathcal{R}_{1,2}^{\infty}$ such that $X^{0}$
and $\bar{a}$ are square integrable random vectors and $E$$\left[\triangle a_{i}\right]$
$=$ $\frac{1}{2},$ for $i$ $=$ $1,$ $2$. If
\[
\Psi_{a}\left(X^{0}\right)\,=\,\prec X^{0},a\succ_{1,2},
\]
taking expectation and multiplying by two it follows that

\[
E\left[X^{0}\cdot(2\bar{a})\right]\,=\, \sup_{\tilde{X}\sim X^{0}}E\left[X\cdot\left(2\bar{a}\right)\right].
\]
Therefore, from \cite{ruschendorf2006law},

\[
E\left[X^{0}\cdot(2\bar{a})\right]\,=\, \sup_{\tilde{X}\sim X^{0},U\sim2\bar{a}}E\left[\tilde{X}\cdot U\right],
\]
and hence

\[
E\left[X^{0}\cdot\bar{a}\right]\,=\, \sup_{\tilde{X}\sim X^{0},U\sim\bar{a}}E\left[\tilde{X}\cdot U\right].
\]
By Theorem 1 in \cite{ruschendorf1990characterization}, we conclude that $X^{0}$ $\in$ $\partial f\left(\bar{a}\right)$,
for some convex lower semi-continuous function $f$. Hence, if $\tilde{X^{1}}$,....,$\tilde{X}^{n}$ are
 $a-comonotone$,  then  $\sum_{i=1}^{n}\tilde{X}^{i}$ $\in$ $\partial f\left(\bar{a}\right)$    and  $\tilde{X}^{i}$  $\in$ $\partial f\left(\bar{a}\right)$   for  $i = 1,...n$.
\end{remark}

Let us recall briefly  the main objective of this paper: The aim  is to propose  natural conditions under  which portfolios with maximum risk preserve such property over time. Ekeland, Galichon and Henry \cite{ekeland2012comonotonic}  considered law invariant, coherent and  lower semi-continuous  risk measures, and introduced strongly coherent measures.  The idea behind these concepts  was somehow    to prevent     unnecessary premium payments to conglomerates as well as to avoid imposing  over-conservative rules to the banks. We follow these ideas,  extending also the results of  Ruschendorf  \cite{ruschendorf2012worst},  defining worst portfolios as those which maximize the aggregated risk over the set of all possible portfolios with the same marginals. The condition imposed,  fixing  the marginals,   is a natural way of formalizing the notion that we are  concerned only  with the aggregate risk and not with its nature.  With this purpose in mind we present now the  definition of worst portfolios.

%Let us  now  use  Definition 3.1 to provide some insight on how decisions are taken, based on a monetary utility function process. First, we  recall that in \cite{ekeland2012comonotonic} the authors consider law invariant, coherent and  lower semi-continuous  risk measures, and introduce strongly coherent measures, which have as objective  to prevent giving an unnecessary premium to conglomerates as well as to avoid imposing an over-conservative rule to the banks. In a more recent development,   the authors in \cite{ruschendorf2012worst} study worst case portfolios, which can be thought as the best case portfolios of the insurance version  associated with a monetary risk measure. These portfolios are best in the sense that their aggregated risk maximizes the insurance version of the utility function, among all  possible aggregated risks, with given marginals. Given that these ideas have only been developed in the static case, we extend them to the dynamic framework studied in this paper, showing some interesting intertemporal connections.

\begin{definition}\label{Def3.2}
Given a monetary utility function $\phi_{t,T}$  and  $\bar{X}^{i}\in \mathcal{R}_{t,T}^{\infty}$ with $\bar{X}^{i}$ $\sim$
$X^{i}$,  $i = 1,\ldots, n$, we say that   $\left(\bar{X}^{1},....,\bar{X}^{n}\right)$ is a worst case portfolio for the associated insurance version $\Psi_{t,T}$ if
%\[
%  {\rm ess}\, \sup_{\tilde{X}^{i}\sim X^{i}}\phi_{t,T}\left(\frac{1}{n}\sum_{i=1}^{n}\tilde{X}^{i}\right)\,=\,\phi_{t,T}\left(\frac{1}{n}\sum_{i=1}^{n}\bar{X}^{i}\right).
%\]
%\textcolor{red}{Similarly, we say that  $\left(\bar{X}^{1},....,\bar{X}^{n}\right)$ is a best case portfolio for the respective insurance version $\Psi_{t,T}$ if}
\[
{\rm ess}\, \sup_{\tilde{X}^{i}\sim X^{i}}\Psi_{t,T}\left(\frac{1}{n}\sum_{i=1}^{n}\tilde{X}^{i}\right)\,=\,\Psi_{t,T}\left(\frac{1}{n}\sum_{i=1}^{n}\bar{X}^{i}\right).
\]

\end{definition}

When $n$ $=$ $1 $  the function $\Psi_{t,T}$ consists in  evaluating  a single financial position and the statement becomes trivial for law invariant measures, see the definition below. This is the  case of the widely used {\it value at risk} measure. For $n>1$, the definition of worst portfolios can be explained as selecting  portfolios with maximum risk,  recalling that the function $\Psi_{t,T}(-\cdot)$ is a risk measure. Thus, defining $\rho_{t,T} : \mathcal{R}_{t,T}^{\infty}$ $\longrightarrow$ $L^{\infty}\left(\mathcal{F}_{t}\right)$ as 
	\[
	    \rho_{t,T}(X) =  \Psi_{t,T}(-X),\,\,\,\,\text{ for each} \; X \in \mathcal{R}_{t,T}^{\infty},
	\]
	we have that $\left(\bar{X}^{1},....,\bar{X}^{n}\right)$ is a worst case portfolio for $\Psi_{t,T}$ if and only if 
	\[
	   {\rm ess}\, \sup_{\tilde{X}^{i}\sim -X^{i}}\rho_{t,T}\left(\frac{1}{n}\sum_{i=1}^{n}\tilde{X}^{i}\right)\,=\,\rho_{t,T}\left(\frac{1}{n}\sum_{i=1}^{n}-\bar{X}^{i}\right).
	\]
	This means  that, among all the portfolios with marginals specified by $-X^1,\ldots,-X^n$, the portfolio $(-\bar{X}^{1},\ldots,-\bar{X}^{n})$ has the highest risk with respect to the dynamic monetary risk measure $\rho_{t,T}$.

%When $n$ $=$ $1 $  and the monetary utility  function is evaluated at a random variable (corresponding to a risky position)  depending only on its distribution, then Definition 3.2 becomes a trivial statement; this would be the case of the widely used value at risk measure. However,  if $n$ $>$ $1$, then the best case portfolio can be though as the best opportunity, in terms of utility, that the investor could have as a result of the aggregated risk. Since we are mainly concerned with  utilities for aggregated risks,  Definition 3.2 can be read as  restricting the supreme  over process with fixed marginals.

As it has been  already mentioned,  it is usual   that investors deal with monetary utility functions, or equivalently, monetary risk measures, that  depend only on the distribution of the portfolio process.   We now present  such  definition in our context.

\begin{definition}
We say that a monetary utility function  $\phi_{t,T}$  is law invariant if for all $X \sim X^{\prime}$ for $X$, $X^{\prime}$ $\in$  $\mathcal{R}_{t,T}^{\infty}$ the following holds
\[
    \phi_{t,T}\left(X\right) = \phi_{t,T} \left(X^{\prime}\right).
\]
\end{definition}

We are now ready to state one of  the main results of this section.  This result connects the two fundamental concepts of comonotonicity and worst case portfolios introduced before.
	In the first part, it is  shown that the value of the worst case portfolios is the same as the value of the average risk function,  introduced in Definition \ref{Def3.1}, at its worst scenario. As a consequence of this result, the second part states  sufficient  conditions for being a worst portfolio, in terms of the comonotonicity property. 
 
%It provides mathematical insights in our discussion of best portflios,  and extends  previous results  presented in \cite{ruschendorf2012worst}} for the static case.

\begin{theorem}\label{Thm3.1}
Let $\phi_{t,T}$ be a  law invariant concave monetary utility process such that it is  continuous for bounded decreasing sequences, and denote by $\Psi_{t,T}$ its  insurance version. Then,
the following properties hold:

$\left(i\right)$

\[
{\rm ess}\, \sup_{\tilde{X}^{i}\sim X^{i}}\Psi_{t,T}\left(\frac{1}{n}\sum_{i=1}^{n}\tilde{X}^{i}\right)\,=\, {\rm ess}\, \sup_{a\in D_{t,T}}F_{t,T}\left(a\right).
\]

$\left(ii\right)$ If $\bar{X}^{1}$,....,$\bar{X}^{n}$
are $a^{0}-comonotone$, with $\bar{X}^{i}$ $\sim$ $X^{i}$ , and
$a^{0}$ is a worst scenario of $F_{t,T}$, then $\left(\bar{X}^{1},....,\bar{X}^{n}\right)$ is a worst case portfolio of $\Psi_{t,T}$.
\end{theorem}

\begin{proof}
From  Theorem 3.16  in \cite{cheridito2006dynamic} the representation
\[
\Psi_{t,T}\left(X\right)\,=\,{\rm  ess}\; \sup_{a\in D_{t,T}}\left\{ \prec X,a\succ_{t,T}+ \phi_{t,T}^{\#}\left(a\right)\right\}
\]
holds. This yields

\[
\begin{array}{lll}
 & & {\rm ess}\, \sup_{\tilde{X}^{i}\sim X^{i}}\Psi_{t,T}\left(\frac{1}{n}\sum_{i=1}^{n}\tilde{X}^{i}\right)\;\\
&= & {\rm ess}\, \sup_{\tilde{X}^{i}\sim X^{i}}{\rm ess}\; \sup_{a\in D_{t,T}}\left\{ \prec\frac{1}{n}\sum_{i=1}^{n}\tilde{X}^{i},a\succ_{t,T}+\phi_{t,T}^{\#}\left(a\right)\right\} \\
&= & {\rm ess}\; \sup_{a\in D_{t,T}}{\rm ess}\, \sup_{\tilde{X}^{i}\sim X^{i}}\left\{ \prec\frac{1}{n}\sum_{i=1}^{n}\tilde{X}^{i},a\succ_{t,T}+\phi_{t,T}^{\#}\left(a\right)\right\} \\
&= & {\rm ess}\; \sup_{a\in D_{t,T}}\left\{ \frac{1}{n}\sum_{i=1}^{n}{\rm ess}\, \sup_{\tilde{X}^{i}\sim X^{i}}\prec\tilde{X}^{i},a\succ_{t,T}+\phi_{t,T}^{\#}\left(a\right)\right\} \\
&= & {\rm ess}\; \sup_{a\in D_{t,T}}F_{t,T}\left(a\right).
\end{array}
\]
This completes the proof of part $\left(i\right)$.
From the above argument, we get that

\[
{\rm ess}\, \sup_{\tilde{X}^{i}\sim X^{i}}\Psi_{t,T}\left(\frac{1}{n}\sum_{i=1}^{n}\tilde{X}^{i}\right)\,=\, F_{t,T}\left(a^{0}\right)\,=\,\frac{1}{n}\sum_{i=1}^{n}{\rm ess}\, \sup_{\tilde{X}^{i}\sim X^{i}}\prec\tilde{X}^{i},a^{0}\succ_{t,T}+\phi_{t,T}^{\#}\left(a^{0}\right).
\]
Since $\phi_{t,T}$ is law invariant, if $X$ $\in$ $\mathcal{R}_{t,T}^{\infty}$,
then

\[
\begin{array}{lll}
\Psi_{t,T}\left(X\right) & = & {\rm ess}\, \sup_{\tilde{X}\sim X}\Psi_{t,T}\left(\tilde{X}\right)\qquad\qquad\qquad\qquad\qquad\qquad\;\,\\
 & = & {\rm ess}\, \sup_{\tilde{X}\sim X}{\rm ess}\; \sup_{a\in D_{t,T}}\left\{ \prec\tilde{X},a\succ_{t,T}+\phi_{t,T}^{\#}\left(a\right)\right\} \\
 & = & {\rm ess}\; \sup_{a\in D_{t,T}}\left\{ \Psi_{a}\left(X\right)+\phi_{t,T}^{\#}\left(a\right)\right\} .\quad\qquad\qquad\qquad\;
\end{array}
\]

The last equation and comonotonicity imply that

\[
\begin{array}{lll}
\Psi_{t,T}\left(\frac{1}{n}\sum_{i=1}^{n}\bar{X}^{i}\right) & \geq & \Psi_{a^{0}}\left(\frac{1}{n}\sum_{i=1}^{n}\bar{X}^{i}\right)+\phi_{t,T}^{\#}\left(a^{0}\right)\\
 & = & \frac{1}{n}\sum_{i=1}^{n}\Psi_{a}\left(\bar{X}^{i}\right)+\phi_{t,T}^{\#}\left(a^{0}\right)\\
 & = & F_{t,T}\left(a^{0}\right),\qquad\qquad\qquad\quad\;
\end{array}
\]
which combined with the previous equations yields $\left(ii\right).$
%\hfill{}$\Box$ $\:$

\end{proof}

The above theorem transforms the problem of studying the worst
case portfolios  into one where
comonotone portfolios of the worst scenarios are analyzed. Moreover,
this result links the problem of studying the worst portfolios in
a dynamic context to the theory of optimal coupling, as stated in
Remark \ref{Rem3.1}.

 Let us now give an example of worst case portfolios in a dynamic setting, that is based on an example provided by  Ruschendorf and Uckelmann \cite{ruschendorf2002n}; similar extensions can  be done  for other  examples presented in that paper.
\begin{example}
Define $\phi_{0,T}$ for $X \in \mathcal{R}_{0,T}^{\infty}$ as

\[
   \phi_{0,T}(X) \,=\,  E\left(1_{C_1}\sum_{t=0}^{T} X_{t}\triangle a_{t}\right).
\]
for some  $a$ $\in$ $A^{1}$ and $C_1 \in \mathcal{F}$ . Then,  we define $X^{i}$ $\in$ $\mathcal{R}_{t,T}^{\infty}$  as

\[
    X^{i} =: 1_{C_1}B_i\,\triangle a + a_i
\]
where $B_i$ is  a $T\times T$ positive semidefinite matrix and $a_i \in \mathbb{R}^{T}$.  Here  the processes $X^{i}$  and $\triangle a$  are written as  column random vectors in $\mathbb{R}^{T}$. Then, from \cite{ruschendorf2002n}, we obtain that $\left(X^{1},...,X^{n}\right)$  is a worst case portfolio of $\Psi_{0,T}$. 
Next we can go one step further, considering $\mathcal{F}_1 = \sigma \left(C_{1}^{1},...,C_{m_1}^{1}\right)$ where $C_{1}^{1},\ldots,C_{m_1}^{1}$
is a partition of $\Omega$. We define  $\phi_{1,T} $  as 

\[
\phi_{1,T}\left(X\right) = 1_{C_1}E\left(\sum_{t=1}^{T} X_{t}\triangle a_{t} |  \mathcal{F}_1\right),
\]
with $C_1 = C_{1}^{1}$. It is not difficult to see \cite{ruschendorf2002n} that  defining $X^{i,1} \in \mathcal{R}_{1,t}^{\infty}$
 as $X_{s}^{i,1} := X_{s}$, for $s= 1,\ldots,T$ and $i=1,...,n$, we obtain that $\left(X^{1,1},X^{2,1},...,X^{n,1}\right)$ is a worst case portfolio of $\Psi_{1,t}$. 

This construction can continue, with $\mathcal{F}_{2} = \sigma\left(C_{1}^{2},...,C_{m_2}^{2}\right)$  such that $\{C_{1}^{2},\ldots,C_{m_2}^{2}\}$ is a partition of $\Omega$, and  

\[
\phi_{2,T}\left(X\right) = 1_{A_2}E\left(\sum_{t=2}^{T} X_{t}\triangle a_{t} |  \mathcal{F}_2\right), 
\]
with $C_2 = C_1$. And just as in the previous period of time $t = 1$,  one can verify that  $\left(X^{1,2},X^{2,2},...,X^{n,2}\right)$ is a worst case portfolio of  $\Psi_{2,t}$  where  $X_{s}^{i,2} := X_{s}^{i}$, for $s= 2,\ldots,T$ and $i=1,...,n$. Finally, proceeding by induction, we obtain that the restriction of $\left(X^{1},...,X^{n}\right)$  to $\mathcal{R}_{t,T}^{\infty}$ is a worst case portfolio of $\Psi_{t,T}$  for all $t=0,...,T$.
  
\end{example}

Before presenting the main result of this section, we illustrate at smaller scale how the property of time consistency and the notion of worst case portfolios  are related. In particular, Theorem \ref{Thm3.2}  provides sufficient  conditions to ensure that the worst case portfolios at time zero will also be worst case portfolios at time one. This is an important fact from the point of view of a  practitioner, since at time zero it would allow  him/her to know  the worst scenario at the next time step.

\begin{theorem}\label{Thm3.2}
Let $\left(\phi_{t,T}\right)_{t\in\left[0,2\right]\cap\mathbb{N}}$
be a time-consistent monetary utility function process such that the insurance version at the initial time satisfies
\[
\Psi_{0,2}\left(X\right)\,=\, max_{Q\in\mathcal{Q}}\left\{ E_{Q}\left[X\right]-\alpha\left(Q\right)\right\} ,
\]
for all $X$ $\in$ $\mathcal{R}_{0,2}^{\infty}$ with $E_{Q}\left[X\right]:=\sum_{t=0}^{2}E_{Q_t}\left[X_t\right]$ for all $Q$ $\in$ $\mathcal{Q}$
$\subset$ $\left\{ \left(Q_{1},Q_{2},Q_{3}\right)\mid Q_{1}\in\mathbb{R}^{+},Q_{2},Q_{3}\in\mathcal{M}_{P}\right\} ,$
where $\mathcal{M}_{P}$ is the set of finitely additive measures
equivalent to $P$. If $\left(\bar{X}^{1},...,\bar{X}^{n}\right)$
is a worst portfolio of $\Psi_{0,2}$ and $\left(\hat{X}^{1},...,\hat{X}^{n}\right)$
is a worst portfolio of $\Psi_{1,2}$ such that
\[
\bar{X}^{i}\sim\hat{Y}^{i}\,:=\,\left(\bar{X}_{0}^{i},\hat{X}_{1}^{i},\hat{X}_{2}^{i}\right),\quad i=1,...,n,
\]
where   $\sim$    denotes equality in distribution of vectors, then,$\left(\bar{X}^{1},...,\bar{X}^{n}\right)$
is a worst portfolio of $\Psi_{1,2}$.
\end{theorem}

\begin{proof}
Let $A$ $\in$$\mathcal{F}_{1}$
be the event where
%\end{spacing}

%\begin{spacing}{0.69999999999999996}
\[
\Psi_{1,2}\left(\frac{1}{n}\sum_{i=1}^{n}\bar{X}^{i}\right)\,<\,\Psi_{1,2}\left(\frac{1}{n}\sum_{i=1}^{n}\hat{X}^{i}\right),
\]
and let us assume that $P\left(A\right)$ $>$$0$. Define the processes
$\bar{Y}$ and $\hat{Y}$ by

\[
\bar{Y}\,=\,\left(\frac{1}{n}\sum_{i=1}^{n}\bar{X}^{i}\right)1_{\{0\}}+\Psi_{1,2}\left(\frac{1}{n}\sum_{i=1}^{n}\bar{X}^{i}\right)1_{[1,\infty)},
\]

\[
\hat{Y}\,=\,\left(\frac{1}{n}\sum_{i=1}^{n}\bar{X}^{i}\right)1_{\{0\}}+\Psi_{1,2}\left(\frac{1}{n}\sum_{i=1}^{n}\hat{X}^{i}\right)1_{[1,\infty)}.
\]
Clearly $\bar{Y}$ $\leq$ $\hat{Y}$, and for $t$ $\geqq$ $1$

\[
\bar{Y_{t}}\,<\,\hat{Y}_{t},\qquad in\; A.
\]
By hypothesis we can choose $Q^{0}$ $\in$ $\mathcal{Q}$ such that

\textsf{
\[
\Psi_{0,2}\left(\bar{Y}\right)\,=\, E_{Q^{0}}\left[\bar{Y}\right]-\alpha\left(Q^{0}\right).
\]
} Since $P\left(A\right)$ $>$ $0,$

\[
\begin{array}{ccc}
\Psi_{0,2}\left(\bar{Y}\right) & < & E_{Q^{0}}\left[\hat{Y}\right]-\alpha\left(Q^{0}\right)\\
 & \leq & \Psi_{0,2}\left(\hat{Y}\right).\qquad\qquad
\end{array}
\]
Finally, time-consistency and the last inequality imply that
\[
\Psi_{0,2}\left(\frac{1}{n}\sum_{i=1}^{n}\bar{X}^{i}\right)\,<\,\Psi_{0,2}\left(\frac{1}{n}\sum_{i=1}^{n}\hat{Y}^{i}\right),
\]
which is a contradiction.
%\hfill{}$\Box$ $\:$
\end{proof}

\begin{remark}Theorem \ref{Thm3.2} provides  insight about the requirements in order that the   property of being a worst portfolio is preserved in  two stages. In fact, the key  assumption is  that the insurance version is a penalized  expectation over a discrete set of lotteries. However, the idea of using a dynamic programming type of argument seems to be a natural venue to provide  a more general result.  In order to accomplish this aim, we shall modify accordingly the hypothesis on the discrete set of lotteries,   which appear on the  representation of the monetary utility function.
\end{remark}
Now we  move our analysis to temporal relations  for a larger time horizon.  With this objective in mind,  we introduce before some relevant concepts.

\begin{definition}\label{Def3.4}
Let $\left(\phi_{t,T}\right)_{t\in\left[0,T\right]}$
be a  monetary utility function process. We call
$\left\{ \left(X^{t,1},...,X^{t,n}\right)\right\} _{t\in\left[0,T\right]\cap\mathbb{N}}$
an adapted worst portfolio process for the insurance version $\Psi_{t,T}$ if  $\left(X^{t,1},...,X^{t,n}\right)$
is a  worst portfolio  of  $\Psi_{t,T}$ for each $t$ $=$ $0$,....,$T,$
and
%$\left(\phi_{t,T}\right)_{t\in\left[0,T\right]}$

\[
X^{t,i}\sim\left(X_{t}^{t,i},X_{t+1}^{t+1,i},...,X_{T}^{t+1,i}\right),\quad i=1,...,n,
\]
for $t=0,\ldots,T-1$. Notice that, abusing of the notation, we are denoting by $X^{t,i}$ the process starting at time $t$ given by
$(X^{t,i}_t,\ldots, X^{t,i}_T)$.
\end{definition}

Definition \ref{Def3.4} can be explained as  follows. An adapted worst portfolio process  consists of worst case portfolios at each period of time such that, when any of the process at a given period of time $t$ is constrained to the subsequent period $t+1$, then it has the same marginals as the worst case portfolio process at  time $t$. 

%\textcolor{red}{The idea of  an adapted worst portfolio process is that of a sequence having  }

Next, we illustrate with an example the connection between adapted worst case portfolio process and time consistency. This will be put in a more broad perspective in the rest of this section. 
%We are interested in studying the relation between best portfolios
%at different periods of time, and this definition goes in that direction;
%it also generalizes to longer time steps one of the hypothesis of
%Theorem 3.2. The idea is that at every step we are able to modify
%our portfolio with the restriction given above.

\begin{example}
Let $\left(\phi_{t,T}\right)_{t\in\left[0,T\right]\cap\mathbb{N}}$
be the utility function process described in Example 2.1. Given \\
$\left\{ \left(X^{t,1},...,X^{t,n}\right)\right\} _{t\in\left[0,T\right]\cap\mathbb{N}}$ an adapted worst
portfolio process of the respective  insurance version process\\
 $\left(\Psi_{t,T}\right)_{t\in\left[0,T\right]\cap\mathbb{N}}$,
we claim that for each $t_{0}$ $\in$ $\left[0,T\right]$ $\cap$
$\mathbb{N}$ fixed, $\left(X^{t_{0},1},...,X^{t_{0},n}\right)$ is
the worst portfolio of \\$\Psi_{t,T}$, for $t$ $=$ $t_{0}$,....,$T$. To verify this, we just have to prove that $\left(X^{t_{0},1},...,X^{t_{0},n}\right)$
is the worst portfolio of $\Psi_{t_{0}+1,T}.$ By hypothesis, we
know that

\[
\Psi_{t_{0}+1,T}\left(\frac{1}{n}\sum_{i=1}^{n}X^{t_{0},i}\right)\leq\Psi_{t_{0}+1,T}\left(\frac{1}{n}\sum_{i=1}^{n}\bar{X}^{t_{0}+1,i}\right),
\]
where $\bar{X}^{t_{0}+1,i}$ $\in$ $\mathcal{R}_{t_{0},T}^{\infty}$,
$i$ $=$ $1,$ $...$ $n,$ with  $\bar{X}_{t_{0}}^{t_{0}+1,i}$$=$ $X_{t_{0}}^{t_{0},i},$ and $\bar{X}_{s}^{t_{0}+1,i}$   $=$
$X_{s}^{t_{0}+1,i},$ for  $s$ $=$ $t_{0}+1,$ $...$ $T.$ Moreover, time-consistency and the definition
of an adapted worst portfolio process yields

\[
\frac{1}{\alpha}\log\, E\left(\exp\left(\alpha\,\Psi_{t_{0}+1,T}\left(\frac{1}{n}\sum_{i=1}^{n}\bar{X}^{t_{0}+1,i}\right)\right)\mid\mathcal{F}_{t_{0}}\right)=
\Psi_{t_{0},T}\left(\frac{1}{n}\sum_{i=1}^{n}\bar{X}^{t_{0}+1,i}\right)\leq
\]

\[
\Psi_{t_{0},T}\left(\frac{1}{n}\sum_{i=1}^{n}X^{t_{0},i}\right)=\frac{1}{\alpha}\log\, E\left(\exp\left(\alpha\,\Psi_{t_{0}+1,T}\left(\frac{1}{n}\sum_{i=1}^{n}X^{t_{0},i}\right)\right)\mid\mathcal{F}_{t_{0}}\right).
\]
From the last two inequalities we conclude that $\left(X^{t_{0},1},...,X^{t_{0},n}\right)$
is a worst portfolio of $\Psi_{t_{0}+1,T}$ and the result follows.
\end{example}

We are now ready to present the main result of this paper.  It puts
	on firm ground the intuitive idea that a worst case portfolio can
	be preserved over time, under suitable conditions. This implies that an agent will face the maximum aggregated risk across time given by the same portfolios. This has the following  important financial implication: In order to find a worst case portfolio,  it is sufficient to do it at the very beginning.  Moreover, it is  proved that there is   strong connection between worst case portfolios across time  and the temporal notion of time consistency, which can be interpreted as a  dynamic programming condition \cite{cheridito2006dynamic}. Notice that this connection between dynamic programming and time consistency of dynamic risk measures has been explored already  in different contexts, and next we show some other useful implications of this relationship.

\begin{theorem}\label{MainThm}
For each $s\in\left[0,T\right]\cap\mathbb{N}$,
let $\mathcal{Q}_{s}$$\subset$$D_{s,T}^{e}$ be a convex set with
$\bigtriangleup a_{k}$$\leq$$\triangle b_{k}$, $\varepsilon_{s}$$\leq$$\sum_{j=t+1}^{T}\triangle a_{j},$
for all $a$ $\in$$\mathcal{Q}_{s}$, $k$ $\in$$\mathbb{N}$, with
$b$ $\in$ $A^{1}$ and $\varepsilon_{s}$$\in$$L_{+}^{\infty}\left(\mathcal{F}_{s}\right)$
such that $P\left(\varepsilon_{s}>0\right)$ $=$ $1$. Now, we define  $\Psi_{s,T}\left(\cdot\right) = -\phi_{s,T}\left(-\cdot\right)$ and assume
that $\left(\phi_{s,T}\right)_{s\in\left[0,T\right]\cap\mathbb{N}}$
is a time-consistent monetary utility function process, such that
%the process $\left(\Psi_{s,T}\right)_{s\in\left[0,T\right]\cap\mathbb{N}}$ with $\Psi_{s,T}\left(\cdot\right) = -\phi_{s,T}\left(-\cdot\right)$ satisfies the representation
\begin{equation}\label{FenchelRep}
\phi_{s,T}\left(X\right)={\rm ess}\: \inf{}_{a\in \mathcal{Q}_{s}}\left\{ \prec X,a\succ_{s,T}-\phi_{s,T}^{\#}\left(a\right)\right\} ,\quad X\in\mathcal{R}_{s,T}^{\infty},
\end{equation}
for all $s=0,1,\cdots,T.$  If $\left\{ \left(X^{s,1},...,X^{s,n}\right)\right\} _{s\in\left[0,T\right]\cap\mathbb{N}}$
is an adapted worst portfolio process of $\left(\Psi_{s,T}\right)_{s\in\left[0,T\right]\cap\mathbb{N}}$,
then $\left(X^{0,1},...,X^{0,n}\right)$ is a worst case portfolio
of $\Psi_{t,T}$, for all $0$ $\leq$ $t$ $\leq$ $T$; see Definition \ref{Def3.2}.
\end{theorem}

\begin{proof}
It is enough to prove that the worst case portfolio at the time $t$ is also a worst case portfolio at time $t+1$. Let $X$ $\in$
$\mathcal{R}_{0,T}^{\infty}$ be fixed. Choose $a$, $b$ $\in$ $\mathcal{Q}_{t}$
arbitrary and define $d$ $=$ $a$$1_{U}$$+$$b$$1_{U^{c}}$$\in$
$\mathcal{Q}_{t}$, with $U$ defined by

\[
U:=\left\{ \prec X,a\succ_{t,T}+\phi_{t,T}^{\#}\left(a\right) \,>\, \prec X,b\succ_{t,T}+\phi_{t,T}^{\#}\left(b\right)\right\} \in\mathcal{F}_{t}.
\]
Therefore,

\[
\max\left\{ \prec X,a\succ_{t,T}+\phi_{t,T}^{\#}\left(a\right),\prec X,b\succ_{t,T}+\phi_{t,T}^{\#}\left(b\right)\right\} \,=\,\prec X,d\succ+\phi_{t,T}^{\#}\left(d\right),
\]
which implies that $\left\{ \prec X,a\succ_{t,T}+\phi_{t,T}^{\#}\left(a\right)\mid a\in\mathcal{Q}_{t}\right\} $
is directed upwards, and hence there exists a sequence $\left\{ a^{k}\right\} _{k\in\mathbb{N}}$
$\subset$ $\mathcal{Q}_{t}$ such that

\[
\lim_{k\longrightarrow\infty}\left\{ \prec X,a^{k}\succ_{t,T}+\phi_{t,T}^{\#}\left(a^{k}\right)\right\} \nearrow\Psi_{t,T}\left(X\right).
\]

Now, define the event $A=\left\{ \Psi_{t+1,T}\left(\frac{1}{n}\sum_{i=1}^{n}X^{t,i}\right)\,<\,\Psi_{t+1,T}\left(\frac{1}{n}\sum_{i=1}^{n}X^{t+1,i}\right)\right\} $,
and we shall assume that $P[A]>0$. Define the processes $\bar{X}^{t}$,
$\bar{X}^{t+1}$$\in$ $\mathcal{R}_{t,T}^{\infty}$ as
\[
\bar{X}_{s}^{t}\,=\,\frac{1}{n}\sum_{i=1}^{n}X_{s}^{t,i},\qquad t\leq s\leq T
\]
\[
\bar{X}_{s}^{t+1}\,=\,\frac{1}{n}\sum_{i=1}^{n}X_{s}^{t+1,i},\qquad t+1\leq s\leq T
\]

\[
\bar{X}_{t}^{t+1}\,=\,\frac{1}{n}\sum_{i=1}^{n}X_{t}^{t,i}.
\]

Time-consistency and the previous arguments imply that there is a
sequence $\left\{ a^{k}\right\} _{k\in\mathbb{N}}$ $\subset$ $\mathcal{Q}_{t}$
such that
\[
\begin{array}{lll}
\Psi_{t,T}\left(\bar{X}^{t}\right) & = & \Psi_{t,T}\left(\left(\bar{X}^{t}\right)1_{\left\{ t\right\} }+\Psi_{t+1,T}\left(\bar{X}^{t}\right)1_{[t+1,\infty)}\right)\\
 & = & \lim_{k\rightarrow\infty}\left\{ \prec\left(\bar{X}^{t}\right)1_{\left\{ t\right\} }+\Psi_{t+1,T}\left(\bar{X}^{t}\right)1_{[t+1,\infty)},a^{k}\succ_{t,T}+\phi_{t,T}^{\#}\left(a^{k}\right)\right\} .
\end{array}
\]
Defining $Y$ $=$ $\left(\bar{X}^{t}\right)1_{\left\{ t\right\} }$$+$
$\Psi_{t+1,T}\left(\bar{X}^{t}\right)1_{[t+1,\infty)},$ for $k$
$\in$$\mathbb{N}$ it follows that
\[
\prec Y,a^{k}\succ_{t,T}\,=\, E\left(\bar{X}_{t}^{t}\triangle a_{t}^{k}+\Psi_{t+1,T}\left(\bar{X}^{t}\right)\left(\sum_{j=t+1}^{T}\triangle a_{j}^{k}\right)\mid\mathcal{F}_{t}\right).
\]
Clearly,

\[
\begin{array}{lll}
Y_{t}^{k} & := & \bar{X}_{t}^{t}\triangle a_{t}^{k}+\Psi_{t+1,T}\left(\bar{X}^{t}\right)\left(\sum_{j=t+1}^{T}\triangle a_{j}^{k}\right)\\
 & \leq & \bar{X}_{t}^{t+1}\triangle a_{t}^{k}+\Psi_{t+1,T}\left(\bar{X}^{t+1}\right)\left(\sum_{j=t+1}^{T}\triangle a_{j}^{k}\right)=:Y_{t+1}^{k}.
\end{array}
\]

For each $k$ $\in$ $\mathbb{N}$ denote by $C_{k}$ the set $\{\sum_{j=t+1}^{T}\triangle a_{j}^{k}>0\}$.
Observe that, since $a^{k}$ $\in$ $D_{0,T}^{e}$, we have that $P\left(C_{k}\right)$
$=$ $1$, and hence the set $C$ $=$ $\cap_{k\in\mathbb{N}}$$C_{k}$
is such that $P\left(C\right)$ $=$$1$. This implies that $P\left(C\cap A\right)$
$=$ $P\left(A\right)$ $>$ $0$. By the definitions given above,
\[
\left\{ Y_{t}^{k}\,<\, Y_{t+1}^{k}\right\} =C_{k}\cap A,\qquad {\rm  for\; all}\; k\in\mathbb{N}.
\]
Let $D$ $:=$$\cap_{k\in\mathbb{N}}$$B_{k}^{c},$ with $B_{k}$
$:=$ $\left\{ E\left[Y_{t}^{k}\mid\mathcal{F}_{t}\right]\,=\, E\left[Y_{t+1}^{k}\mid\mathcal{F}_{t}\right]\right\} $$\in$$\mathcal{F}_{t}$,
and observe that

\[
\begin{array}{lll}
0 & = & \int_{B_{k}}\left(E\left(Y_{t+1}^{k}\mid\mathcal{F}_{t}\right)-E\left(Y_{t}^{k}\mid\mathcal{F}_{t}\right)\right)dP\\
 & = & \int_{B_{k}}\left(Y_{t+1}^{k}-Y_{t}^{k}\right)dP\\
 & = & \int_{B_{k}\cap\left\{ Y_{t+1}^{k}>Y_{t}^{k}\right\} }\left(Y_{t+1}^{k}-Y_{t}^{k}\right)dP.\qquad\quad
\end{array}
\]

Then,

\[
P\left(B_{k}\cap C\cap A\right)=P\left(B_{k}\cap A\right)=P\left(B_{k}\cap C_{k}\cap A\right)=0,
\]
and,

\[
P\left(D^{c}\cap A\right)=P\left(D^{c}\cap C\cap A\right)=0.
\]
On the other hand, by time-consistency and hypothesis, we deduce that

\[
\begin{array}{lll}
\Psi_{t,T}\left(\bar{X}^{t}\right) & = &
\lim_{k\rightarrow\infty}\left\{  E\left[Y_{t}^{k}\mid\mathcal{F}_{t}\right]+\phi_{t,T}^{\#}\left(a^{k}\right)\right\}   \\
 & = & \lim_{k\rightarrow\infty}\left\{ E\left[Y_{t+1}^{k}\mid\mathcal{F}_{t}\right]+\phi_{t,T}^{\#}\left(a^{k}\right)\right\} \\
 & = & \Psi_{t,T}\left(\bar{X}^{t+1}\right),
\end{array}
\]
and hence
\[
\lim_{k\rightarrow\infty}\mid\left(E\left[Y_{t}^{k}\mid\mathcal{F}_{t}\right]+\phi_{t,T}^{\#}\left(a^{k}\right)\right)-\left(E\left[Y_{t+1}^{k}\mid\mathcal{F}_{t}\right]+\phi_{t,T}^{\#}\left(a^{k}\right)\right)\mid=0.
\]
Therefore,

\[
\lim_{k\rightarrow\infty}1_{D}E\left[\left(\sum_{j=t+1}^{T}\triangle a_{j}^{k}\right)\left(\Psi_{t+1,T}\left(\bar{X}^{t+1}\right)-\Psi_{t+1,T}\left(\bar{X}^{t}\right)\right)\mid\mathcal{F}_{t}\right]=0,
\]
which implies

\[
\lim_{k\rightarrow\infty}E\left[1_{D}\left(\sum_{j=t+1}^{T}\triangle a_{j}^{k}\right)\left(\Psi_{t+1,T}\left(\bar{X}^{t+1}\right)-\Psi_{t+1,T}\left(\bar{X}^{t}\right)\right)\right]=0,
\]
since $\left(\sum_{j=t+1}^{T}\triangle a_{j}^{k}\right)$ $\leq$
$\left(\sum_{j=t+1}^{T}\triangle b_{j}^{k}\right)$ for each $k$
$\in$$\mathbb{N}$. Thus,

\[
\lim_{k\rightarrow\infty}E\left[1_{D\cap A\cap C}\left(\sum_{j=t+1}^{T}\triangle a_{j}^{k}\right)\left(\Psi_{t+1,T}\left(\bar{X}^{t+1}\right)-\Psi_{t+1,T}\left(\bar{X}^{t}\right)\right)\right]=0.
\]
Taking a subsequence if necessary, we conclude that

\[
\varepsilon_{s}1_{A\cap D\cap C}\leq \lim_{k\rightarrow\infty}1_{D\cap A\cap C}\sum_{j=t+1}^{T}\triangle a_{j}^{k}=0,
\]
which is a contradiction, since $P$$\left(A\cap D\cap C\cap\left\{ \varepsilon_{s}>0\right\} \right)$ $=$ $P$$\left(A\cap D\right)$ $=$ $P$$\left(A\right)$$>$$0$.
%\hfill{}$\Box$
\end{proof}
\begin{remark}
\begin{enumerate}
\item The main hypothesis in Theorem \ref{MainThm}  is concerned with the  constraint on the set $\mathcal{Q}_{s}$ appearing in the dual representation (\ref{FenchelRep}), which should    be bounded both from above and below. Observe that the first of this conditions,  concerning the existence of $b \in A^1$ satisfying  $\bigtriangleup a_{k}$$\leq$$\triangle b_{k}$ 	for all $a$ $\in$$\mathcal{Q}_{s}$, is not so restrictive,  since  the set $\mathcal{Q}_{s}$ is a subset of the class of positive conditional densities $D_{s,T}^e$, while the set $A^1$ encompasses all the possible normalized conditional measures to the space $\mathcal{R}^{\infty}$. Moreover,  recall that when $a \in D^{e}_{s,T}$, by our definition in the previous  section, we have
\begin{equation}
\label{aux1}
 P\left(\sum_{j = s}^T\bigtriangleup a_{j}>0\right)\,\,=\,\,1.
\end{equation}
 Hence,  we are asking that the increments of the  elements of  $\mathcal{Q}_{s}$ are uniformly away from zero, ensuring that there is a positive effect at each time step, along the evolution of the process. 
 % between the elements of the . that there exists $\varepsilon_{s}$$\in$$L_{+}^{\infty}\left(\mathcal{F}_{s}\right)$  for which $\varepsilon_{s}$$\leq$$\sum_{j=t+1}^{T}\triangle a_{j},$
%	for all $a$ $\in$$\mathcal{Q}_{s}$ can be thought as a stronger a stronger assumption than (\ref{aux1}).

\item It is worth noting the connection of Theorem 3.3 with the necessary conditions for  time consistency presented  in  \cite[Theorem 4.19]{cheridito2006dynamic}. Following the proof of such result, we obtain that for any $t$ and $s$  with   $0 \leq t \leq s \leq T$, the penalty function $\phi_{t,T}^{\#}$  satisfies
\[
  \phi_{t,T}^{\#}\left(a\right)  \,\,=\,\,    {\rm ess}\: \sup{}_{b\in \mathcal{Q}_{s}}   \,\,\phi_{t,T}^{\#}\left(a\oplus_{\Omega}^{\theta}b\right)\,\, + \,\, E\left[ \phi_{s,T}^{\#}(a) \mid  \mathcal{F}_{t}\right],\,\,\,\,\,  \forall  a \in \mathcal{Q}_{t}.
\]
Here the expression  $a\oplus_{\Omega}^{\theta}b$ refers to the concatenation of processes, presented below as  Definition 4.1 in the next section.  The above display can be thought as the time consistency property for the underlying process corresponding to the penalties.
\item	The above theorems generalized the comonotonicity results from \cite{ruschendorf2012worst} to the dynamic framework. 
%Hence, in general settings, 
As a general conclusion, we can say that the property of comonotonicity  seems to be  the next step to characterize worst case portfolios. Summarizing, we have shown in Theorem \ref{Thm3.1} that the problem of worst case portfolios can be translated into a problem of comonotonicity, while Theorem \ref{MainThm} presents a setting under which worst case portfolios can be time invariant.
\end{enumerate}
\end{remark}
%a\oplus_{A}^{\theta}b

%\textcolor{red}{In our path to understanding worst case portfolios of monetary utility functions,  we highlight the contributions we have made so far with Theorems \ref{Thm3.1} and 
%	\ref{MainThm}. The purpose of the former was to state a general result on worst case portfolios for law invariant concave monetary utility process. This 

While the conditions in Theorem  \ref{MainThm} are sufficient to ensure the preservation of worst case portfolios as time evolves, it will be shown below that such conditions are not necessary. This is illustrated  by considering a particular  class of dynamic monetary utility functions with  their corresponding insurance versions. Before  such  example is presented we introduce a technical definition.

\begin{definition}
Let $D_{T}^{rel}$ be the class
\[
D_{T}^{rel}=\left\{ h\in L^{1}\left(\mathcal{F}_{T}\right)\mid h>0,E\left(h\right)=1\right\} .
\]
For $f,$ $g$ $\in$ $D_{T}^{rel},$ $s$ $\in$ $\left[0,T\right]$$\cap$
$\mathbb{N}$ and $A$ $\in$ $\mathcal{F}_{s},$ we define the pasting
$f$ $\otimes_{A}^{s}$ $g$ by

\[
f\otimes_{A}^{s}g:=\begin{cases}
\begin{array}{cc}
f & on\; A^{c}\cup\left[E\left(g\mid\mathcal{F}_{s}\right)=0\right]\\
E\left[f\mid\mathcal{F}_{s}\right]\frac{g}{E\left[g\mid\mathcal{F}_{s}\right]} & on\; A\cap\left[E\left(g\mid\mathcal{F}_{s}\right)>0\right]
\end{array} & .\end{cases}
\]
A subset $\mathcal{P}$ of $D_{T}^{rel}$ is m1-stable
if it contains $f$ $\otimes_{A}^{s}$ $g$ for all $f,$ $g$ $\in$
$\mathcal{P},$ every $s$ $\in$ $\left[0,T\right]$ $\cap$ $\mathbb{N}$
and $A$ $\in$ $\mathcal{F}_{s}.$
\end{definition}

We  conclude this section  with an example that  leaves as an open question the problem of finding necessary and sufficient conditions on a insurance version process to have the same worst case portfolios across time. Thus, although the conditions in Theorem \ref{MainThm}  ensure this property, we proceed to illustrate that the same conclusions can also be achieved with another set of assumptions. Within a more general framework, we shall see  that using the concept of stability under concatenation within subsets of the density processes $D_{0,T}^{e}$ we can achieve this property for a certain class of utility function processes; see Theorem \ref{Thm 4.2}.

\begin{example}
Let $\mathcal{P}$ be
a non-empty m1-stable subset of $D_{T}^{rel}$ and $\alpha$ $>$
$0.$ For $t$ $=$ $0,$$...,$ $T$ and $X$ $\in$ $\mathcal{R}_{t,T}^{\infty},$
define

\[
\phi_{t,T}\left(X\right):=ess\, \inf_{f\in\mathcal{P}}\left\{ -\frac{1}{\alpha}\log\,\frac{E\left(f\, \exp(-\alpha X_{T})\mid\mathcal{F}_{t}\right)}{E\left(f\mid\mathcal{F}_{t}\right)}\right\} ,\quad X\in\mathcal{R}_{0,T}^{\infty}.
\]
 This is a time-consistent utility function process \cite{cheridito2006dynamic},
which is a robust version of the utility function process given in
Example 2.1. We consider the insurance version process $\Psi_{t,T}(\cdot) = -\phi_{t,T}\left( - \cdot\right)$. With this example we intend to illustrate that, even
without the representation in the hypothesis of Theorem 3.3,
it is possible to obtain the same conclusion. Namely, if $\left\{ \left(X^{t,1},...,X^{t,n}\right)\right\} _{t\in\left[0,T\right]\cap\mathbb{N}}$
is an adapted worst case portfolio process of $\left(\Psi_{t,T}\right)_{t\in\left[0,T\right]\cap\mathbb{N}}$,
then $\left(X^{0,1},...,X^{0,n}\right)$ is a worst case  portfolio
of $\Psi_{t,T}$ for all $0$ $\leq$ $t$ $\leq$ $T$. To this end for $r \in \mathbb{N}$ we introduce the notation

\[
   f_r  = \frac{f}{E(f\mid \mathcal{F}_r)}
\]
and fix $t_0$ $\in$ $\{0,1,...,T\}$. Given $X$ $\in$ $\mathcal{R}_{0,T}^{\infty},$ there exists
a sequence, see Section 5.6 in \cite{cheridito2006dynamic}, $\left\{ f^{k}\right\} _{k\in\mathbb{N}}$$\subset$
$D_{T}^{rel}$ such that
\[
\left\{ \frac{1}{\alpha}\log\, E\left(f_{t_{0}}^{k}\, \exp(\alpha X_{T})\mid\mathcal{F}_{t_{0}}\right)\right\} \nearrow\Psi_{t_{0},T}(X),\quad as\; k\longrightarrow\infty,
\]
which implies

\[
\left\{ E\left(f_{t_{0}}^{k}\, \exp(\alpha X_{T})\mid\mathcal{F}_{t_{0}}\right)\right\} \nearrow \exp\left(\alpha\,\Psi_{t_{0},T}(X)\right),\quad as\; k\longrightarrow\infty.
\]
 Now we follow the same lines as in the proof of Theorem 3.3.
First, define the event $$A=\left\{ \Psi_{t_{0}+1,T}\left(\frac{1}{n}\sum_{i=1}^{n}X^{t_{0},i}\right)\,<\,\Psi_{t_{0}+1,T}\left(\frac{1}{n}\sum_{i=1}^{n}X^{t_{0}+1,i}\right)\right\}, $$
and assume that $P[A]$ $>$ $0$. The processes $\bar{X}^{t_{0}}$,
$\bar{X}^{t_{0}+1}$$\in$ $\mathcal{R}_{t_{0},T}^{\infty}$ are defined
so as $\bar{X}^{t}$, $\bar{X}^{t+1}$ in the proof of Theorem 3.3.
Time-consistency and the previous arguments imply that there is a
sequence $\left\{ a^{k}\right\} _{k\in\mathbb{N}}$ $\subset$ $D_{T}^{rel}$
such that

\[
\begin{array}{lll}
\exp\left(\alpha\,\Psi_{t_{0},T}\left(\bar{X}^{t_{0}}\right)\right) & \nwarrow & \lim_{k\rightarrow\infty}\left\{ E\left(a_{t_{0}}^{k}\, \exp(\alpha\,\Psi_{t_{0}+1,T}\left(\bar{X}^{t_{0}}\right))\mid\mathcal{F}_{t_{0}}\right)\right\} \\
 & = & \lim_{k\rightarrow\infty}\left\{ E\left(a_{t_{0}}^{k}\, \exp(\,\alpha\Psi_{t_{0}+1,T}\left(\bar{X}^{t_{0}+1}\right))\mid\mathcal{F}_{t_{0}}\right)\right\} \\
 & = & \exp\left(\alpha\,\Psi_{t_{0},T}\left(\bar{X}^{t_{0}+1}\right)\right).
\end{array}
\]
Therefore,

\[
\lim_{k\rightarrow\infty}\left\{ E\left(a_{t_{0}}^{k}\, \exp(\alpha\,\Psi_{t_{0}+1,T}\left(\bar{X}^{t_{0}+1}\right))\mid\mathcal{F}_{t_{0}}\right)-E\left(a_{t_{0}}^{k}\, \exp(\alpha\,\Psi_{t_{0}+1,T}\left(\bar{X}^{t_{0}}\right))\mid\mathcal{F}_{t_{0}}\right)\right\} =0.
\]
Since for all $k$ $\in$ $\mathbb{N},$

\[
\begin{array}{ccc}
E\left(a_{t_{0}}^{1}\, \exp(\alpha\,\Psi_{t_{0}+1,T}\left(\bar{X}^{t_{0}}\right))\mid\mathcal{F}_{t_{0}}\right) & \leq & E\left(a_{t_{0}}^{k}\, \exp(\alpha\,\Psi_{t_{0}+1,T}\left(\bar{X}^{t_{0}}\right))\mid\mathcal{F}_{t_{0}}\right)\quad\\
 & \leq & E\left(a_{t_{0}}^{k}\, \exp(\alpha\,\Psi_{t_{0}+1,T}\left(\bar{X}^{t_{0}+1}\right))\mid\mathcal{F}_{t_{0}}\right)\\
 & \leq & \exp\left(\alpha\,\Psi_{t_{0},T}\left(\bar{X}^{t_{0}+1}\right)\right),\hspace{1em}\quad\quad\quad\quad\;\:
\end{array}
\]
we have that

\[
\lim_{k\rightarrow\infty}\left\{ E\left(\mid a_{t_{0}}^{k}\, \exp(\alpha\,\Psi_{t_{0}+1,T}\left(\bar{X}^{t_{0}+1}\right))-a_{t_{0}}^{k}\, \exp(\alpha\,\Psi_{t_{0}+1,T}\left(\bar{X}^{t_{0}}\right)\mid)\right)\right\} =0,
\]
and

\[
\lim_{k\rightarrow\infty}\left\{ E\left(\mid \exp\left(\alpha\,\Psi_{t_{0},T}\left(\bar{X}^{t_{0}}\right)\right)-a_{t_{0}}^{k}\, \exp(\alpha\,\Psi_{t_{0}+1,T}\left(\bar{X}^{t_{0}}\right))\mid\right)=0\right\} .
\]
By the last two identities we conclude that there is a subsequence $\left\{ a^{k_{i}}\right\} _{i\in\mathbb{N}}$
$\subset$ $\left\{ a^{k}\right\} _{k\in\mathbb{N}}$ for which

\[
\lim_{i\rightarrow\infty}a_{t_{0}}^{k_{i}}=0\qquad on\quad A,\quad and
\]

\[
\lim_{i\rightarrow\infty}a_{t_{0}}^{k_{i}}\, \exp(\alpha\,\Psi_{t_{0}+1,T}\left(\bar{X}^{t_{0}}\right))=\exp\left(\alpha\,\Psi_{t_{0},T}\left(\bar{X}^{t_{0}}\right)\right).
\]
We conclude that $P$$\left(\exp\left(\alpha\Psi_{t_{0},T}\left(\bar{X}^{t_{0}}\right)\right)=0\right)$
$>$ $0,$ but this is a contradiction.

\end{example}

\section{ Optimization in representation of insurance version functions}

%\begin{spacing}{0.69999999999999996}
The results presented in the previous section are based on representation
	of conditional utility functions in terms of penalty functions,  which can be thought as Fenchel conjugate representations; see (\ref{FenchelRep}).  Now we revisit this type of representations and analyze conditions under which the maximum is attained. Using this  representations, we will show a different venue, to Theorem \ref{MainThm},  in order to  attain the property of stability of worst case portfolios of time consistent insurance versions of utility processes. The main difference with  Theorem \ref{MainThm} is that in this section we are  concerned with insurance versions associated with  coherent, relevant and time-consistent monetary utility function process. Such additional constraints will allow us to  relax some of the assumptions of Theorem \ref{MainThm}. 
Furthermore,   we will also obtain as byproduct some  properties in financial risk that can be deduced  %implications of 
	%also use the same ideas   of 
when  the maximum in the  Fenchel conjugate representations of conditional utility functions is attained.
	%to present different results concerning elementary properties  of financial risk. 

%The results presented in the previous section are based on representation
%	of conditional utility functions in terms of penalty functions,  which can be thought as Fenchel conjugate representations; see (\ref{FenchelRep}).  Now we revisit this type of representations and analyze conditions under which the maximum is attained. This is then used in the  presentation of worst case portfolio properties for some particular cases. 
   
We begin by defining the concatenation  of density processes from a given  period of time to another. This condition was introduced originally by  Cheridito et. al. \cite{cheridito2006dynamic}; see for instance their Theorem 3.16.

\begin{definition}
Let $a,$ $b$ $\in$
$\mathcal{A}_{+}^{1},$ $\theta$ a finite $\left(\mathcal{F}_{t}\right)-$stopping
time and $A$ $\in$ $\mathcal{F_{\theta}}.$Then the concatenation
$a$ $\oplus_{A}^{\theta}$ $b$ is defined by

\[
\left(a\oplus_{A}^{\theta}b\right)_{t}:=1_{C}a_{t}+1_{C^{c}}\left(a_{\theta-1}+\frac{\prec1,a\succ_{\theta,\infty}}{\prec1,b\succ_{\theta,\infty}}\left(b_{t}-b_{\theta-1}\right)\right),
\]
with $C$ $=$ $\left\{ t<\theta\right\} $ $\cup$$A^{c}$
$\cup$$\left\{ \prec1,b\succ_{\theta,\infty}=0\right\}$. A
subset $\mathcal{M}$ of $\mathcal{A}_{+}^{1}$is said to be stable
under concatenation if $a$ $\oplus_{A}^{\theta}$ $b$ $\in$$\mathcal{M}$
for all $a,$$b$ $\in$ $\mathcal{M}$, each $\left(\mathcal{F}_{t}\right)-$stopping
time $\theta$, and any $A$ $\in$ $\mathcal{F}_{\theta}$.
\end{definition}

The next theorem provides topological conditions on the set involved in the
representation theorem to achieve the maximum.   The main assumption is concerned with  the compactness of the representation set with respect to the norm $\parallel\cdot\parallel_{A^{1}}$. Although it is a restrictive assumption, we   impose it  for technical reasons which we were not able to overcome.

\begin{theorem}\label{Thm4.1}
For each $X$ $\in$ $\mathcal{R}_{s,T}^{\infty},$
let $\phi\left(X\right)$ $=$ ${\rm ess}\, \inf_{a\in\mathcal{M\subset}D_{s,T}}\left\{ \prec X,a\succ_{s,T}-\phi^{\#}\left(a\right)\right\} $
be a monetary utility function with $s$ $\in$ $\mathbb{N},$ $s$ $\leq$ $T$. If $\mathcal{M}$
is $\parallel\cdot\parallel_{A^{1}}$-compact, stable under concatenation,
with $\triangle a_{t}$ $\leq$ $\triangle b_{t}$ for all $a$ $\in$
$\mathcal{M},$ $t$ $\in$ $\mathbb{N},$ for some $b$ $\in$ $\mathcal{A}_{+}^{1},$
then, the insurance version  $\Psi\left(\cdot\right) := -\phi\left(-\left(\cdot\right)\right)$ satisfies

\[
\Psi\left(X\right)\,=\,\,\prec X,a^{X}\succ_{s,T}+\phi^{\#}\left(a^{X}\right),
\]
with $a^{X}$ $\in$ $\mathcal{M}$ depending of $X$.
\end{theorem}
\begin{proof}
Given $X$ $\in$ $\mathcal{R}_{s,T}^{\infty}$,
let $\left\{ a^{k}\right\} $ be a sequence in $\mathcal{M}$ such
that $a^{k}$ $\longrightarrow_{\parallel\cdot\parallel_{\mathcal{A}^{1}}}$$a^{*},$
with $a^{*}$$\in$ $\mathcal{M}$.  Passing to
some subsequence if necessary, we have that

\[
\prec X,a^{k}\succ_{s,T}\,\longrightarrow\,\prec X,a^{*}\succ_{s,T}.
\]
Since the last identity holds for all $X$ $\in$ $\mathcal{R}_{s,T}^{\infty},$
it follows that

\[
\overline{\lim}{}_{k\longrightarrow\infty}\phi^{\#}\left(a^{k}\right)\leq\phi^{\#}\left(a^{*}\right),
\]
and hence

\[
\overline{\lim}{}_{k\longrightarrow\infty}\left(\prec X,a^{k}\succ_{s,T}+\,\phi^{\#}\left(a^{k}\right)\right)\leq\,\prec X,a^{*}\succ_{s,T}+\,\phi^{\#}\left(a^{*}\right).
\]

Now we prove that $\left\{ \prec X,a\succ_{s,T}+\,\phi^{\#}\left(a\right)\mid a\in\mathcal{M}\right\} $
is directed upwards, which together with the previous arguments yield
the result. Given $b$, $c$ $\in$ $\mathcal{M}$, from the stability
under concatenation of $\mathcal{M},$ $d$ $=$ $b$ $\oplus_{A}^{s}$
$c$ $\in$ $\mathcal{M}$, with $A$ given by
%\end{spacing}

\[
A=\left\{ \prec X,c\succ_{s,T}+\,\phi^{\#}\left(c\right)>\prec X,b\succ_{s,T}+\,\phi^{\#}\left(b\right)\right\} ,
\]
and

%\begin{spacing}{0.69999999999999996}
\[
d_{r}\,=\, b_{r}1_{A^{C}\cup\left\{ r<s\right\} }+\left(b_{s-1}+\left(c_{r}-c_{s-1}\right)\right)1_{A\cap\left\{ r\geq s\right\} }.
\]
This implies that

\[
\begin{array}{lll}
\prec X,d\succ_{s,T}+\phi^{\#}\left(d\right) & = & 1_{A^{c}}\prec X,b\succ_{s,T}+1_{A}\prec X,c\succ_{s,T}+\,\phi^{\#}\left(1_{A^{c}}b+1_{A}c\right)\\
 & = & 1_{A^{c}}\left(\prec X,b\succ_{s,T}+\,\phi^{\#}\left(b\right)\right)+1_{A}\left(\prec X,c\succ_{s,T}+\,\phi^{\#}\left(c\right)\right)\\
 & = & \max\left\{ \prec X,b\succ_{s,T}+\,\phi^{\#}\left(b\right),\prec X,c\succ_{s,T}+\,\phi^{\#}\left(c\right)\right\} ,
\end{array}
\]
and the theorem follows. 
\end{proof}

Our next result is  an application of Theorem \ref{Thm4.1}, and shows that  worst case portfolios  do not necessary consists in imposing constrains in  the marginal distributions. The aim is to illustrate that the general mathematical properties of worst case portfolios can be analyzed using Theorem \ref{Thm4.1}. Later we shall  go back to the definition of worst case portfolios given in Definition \ref{Def3.2}.   
	%As a consequence of the last theorem,   if there exists a
	%set of strategies to modify a given portfolio, then, under certain
	%conditions, there exists one of those strategies which is safest.}

\begin{corollary}\label{Cor4.1}
Let $\Psi$ be an insurance version satisfying
the conditions of Theorem 4.1, with $s$ $=$ $1$, $T$ $=$ $2$,
and $\phi^{\#}\left(a\right)>-\infty,$ for all $a$ $\in$ $\mathcal{M}$.
Let us assume that $\mathcal{C}$ $\subset$ $Mat_{2\times2}\left(\mathcal{F}_{1}\right)$,
$\left\{ \left(A_{11},A_{12},A_{21},A_{22}\right)\mid A\in\mathcal{C}\right\} $
is a compact set in $\left(L^{\infty}\left(\mathcal{F}_{1}\right)\right)^{4}$,
and
\[
K_{X}=\left\{ \prec AX,a\succ_{1,2}+\phi^{\#}\left(a\right)\mid A\in\mathcal{C},a\in\mathcal{M}\right\}
\]
is directed upwards for all $X$ $\in$ $\mathcal{R}_{1,2}^{\infty}.$
If $X$ $\in$ $\mathcal{R}_{1,2}^{\infty},$ then,
\[
{\rm ess}\, \sup_{A\in\mathcal{C}}\Psi\left(AX\right)=\Psi\left(A^{0}X\right),
\]
for some $A^{0}$ $\in$ $\mathcal{C}$.

\end{corollary}

\begin{proof}
 Let us fix $X$ $\in$ $\mathcal{R}_{1,2}^{\infty}.$ If $A,$
$B$ $\in$ $\mathcal{C},$ Theorem 4.1 implies that

\[
\Psi\left(AX\right)=\:\prec AX,a\succ_{1,2}+\phi^{\#}\left(a\right),
\]

\[
\Psi\left(BX\right)=\:\prec BX,b\succ_{1,2}+\phi^{\#}\left(b\right),
\]
for some $a,$ $b$ $\in$ $\mathcal{M}$. Since $K_{X}$ is directed
upwards, there are $C$ $\in$ $\mathcal{C}$ and $c$ $\in$ $\mathcal{M}$
such that

\[
\prec CX,c\succ_{1,2}+\phi^{\#}\left(c\right)\geq \max\left\{ \prec AX,a\succ_{1,2}+\phi^{\#}\left(a\right),\prec BX,b\succ_{1,2}+\phi^{\#}\left(b\right)\right\} ,
\]
and by Theorem \ref{Thm4.1},

\[
\prec CX,d\succ_{1,2}+\phi^{\#}\left(d\right)=\Psi\left(CX\right)={\rm ess}\, \sup_{e\in\mathcal{M}}\left\{ \prec CX,e\succ_{1,2}+\phi^{\#}\left(e\right)\right\} \geq\prec CX,c\succ_{1,2}+\phi^{\#}\left(c\right),
\]
for some $d$ $\in$ $\mathcal{M}$. Consequently, $\left\{ \Psi\left(AX\right)\mid A\in\mathcal{C}\right\} $
is directed upwards, and hence there is a sequence $\left\{ A^{k}\right\} _{k\in\mathbb{N}}$
$\subset$ $\mathcal{C}$ such that

\[
\quad as \quad k\rightarrow\infty,\quad \quad \Psi\left(A^{k}X\right)\nearrow {\rm ess}\, \sup_{A\in\mathcal{C}}\Psi\left(AX\right).
\]

By compactness of $\left\{ \left(A_{11},A_{12},A_{21},A_{22}\right)\mid A\in\mathcal{C}\right\} $,
we can find $A^{0}$$\in$ $\mathcal{C}$ such that
\[
A_{ij}^{k}\longrightarrow_{k\rightarrow\infty}A_{ij}^{0},\qquad a.s\quad {\rm for\; all}\; i,j\in\left\{ 1,2\right\} ,
\]
passing through a subsequence if necessary. Given $a$ $\in$ \textsf{\emph{$\mathcal{M}$}},
 we have that

\[
 \begin{array}{lll}
\mid\prec A^{k}X,a\succ_{1,2}+\phi^{\#}\left(a\right)-\left(\prec A^{0}X,a\succ_{1,2}+\phi^{\#}\left(a\right)\right)\mid & = & \mid\prec\left(A^{k}-A^{0}\right)X,a\succ_{1,2}\mid\qquad\qquad\quad\:\:\\
 & \leq & 2\parallel X\parallel_{\infty}\max\mid A_{ij}^{k}-A_{ij}^{o}\mid\longrightarrow0\\
 \end{array}
 \]

Therefore,

\[
\mid {\rm ess}\, \sup_{a\in\mathcal{\mathcal{M}}}\left\{ \prec A^{k}X,a\succ_{1,2}+\phi^{\#}\left(a\right)\right\} -{\rm ess}\, \sup_{a\in\mathcal{\mathcal{M}}}\left\{ \prec A^{0}X,a\succ_{1,2}+\phi^{\#}\left(a\right)\right\} \mid\longrightarrow_{k\rightarrow\infty}0,
\]
and the claim follows. 

\end{proof}

The following theorem links some implications  of attaining the maximum
in the robust representation (\ref{FenchelRep}) with the notion of worst case 
portfolios, as presented in Section \ref{Section3}. Conclusions of this theorem
are  close to those of Theorem \ref{MainThm}. Namely, we prove that for a certain class of dynamic utility functions process, the worst case portfolios of the insurance version process are preserved over time.
Despite the similarity of our next result with Theorem \ref{MainThm}, we now impose slightly
different assumptions. The requirement of boundedness from below
in Theorem \ref{MainThm} is now replaced by coherency and stability under
concatenation. Interested readers in  the relationship between
relevance, coherency, time-consistency and the representation given
in Corollary \ref{Cor4.1} are referred to Corollary 4.16 from \cite{cheridito2006dynamic}.

\begin{theorem}\label{Thm 4.2}
Let $\left(\phi_{s,T}\right)_{s\in\left[0,T\right]\cap\mathbb{N}}$
be a coherent, relevant and time-consistent monetary utility function process
such that
\[
\phi_{s,T}\left(X\right)={\rm ess}\: \inf{}_{a\in\mathcal{M}}\frac{\prec X,a\succ_{s,T}}{\prec1,a\succ_{s,T},}\quad X\in\mathcal{R}_{s,T}^{\infty},
\]
for some   compact, convex set $\mathcal{M}$$\subset$$D_{0,T}^{e}$ stable under concatenation,
with $\bigtriangleup a_{k}$$\leq$$\triangle b_{k}$, for all $a$
$\in$$\mathcal{M},$ $k$ $\in$$\mathbb{N},$ with $b$ $\in$$A^{1}.$
If $\left\{ \left(X^{s,1},...,X^{s,n}\right)\right\} _{s\in\left[0,T\right]\cap\mathbb{N}}$
is an adapted worst portfolio process of the respective insurance version  $\left(\Psi_{s,T}\right)_{s\in\left[0,T\right]\cap\mathbb{N}},$
then $\left(X^{0,1},...,X^{0,n}\right)$ is a worst portfolio
of $\Psi_{t,T}$ for all $0$ $\leq$ $t$ $\leq$ $T$.
\end{theorem}

\begin{proof}
%Again, we consider the monetary utility functions defined as $\Psi_{s,T}\left(X\right) := -\phi_{s,T}\left(-X\right)$ for all $X$ $\in$ $\mathcal{R}_{s,T}^{\infty}$, $s=0,1,\ldots,T$.
The first step is just to verify that for each $X$ $\in$ $\mathcal{R}_{0,T}^{\infty},$
  there is  $a$ $\in$
$\mathcal{M},$ such that

\[
\phi_{s,T}\left(X\right)=\frac{\prec X,a\succ_{s,T}}{\prec1,a\succ_{s,T}},\;\;\;\;s\in [0,T] \cap\mathbb{N}.
\]
This can be done  following similar arguments as in the proof  of Theorem 4.1

Hence, it is  enough to prove that, for  each $t$ $\in$ $[0,T-1]$ ,
$\left(X^{t,1},...,X^{t,n}\right)$ is a worst case portfolio of $\Psi_{t+1,T}.$
Define the event $$A=\left\{ \Psi_{t+1,T}\left(\frac{1}{n}\sum_{i=1}^{n}X^{t,i}\right)\,<\,\Psi_{t+1,T}\left(\frac{1}{n}\sum_{i=1}^{n}X^{t+1,i}\right)\right\}, $$
and assume that $P[A]$ $>$ $0.$   Let $\bar{X}^{t}$,
$\bar{X}^{t+1}$$\in$ $\mathcal{R}_{t,T}^{\infty}$ be the processes defined as

\[
\bar{X}_{s}^{t}\,=\,\frac{1}{n}\sum_{i=1}^{n}X_{s}^{t,i},\qquad t\leq s\leq T,
\]

\[
\bar{X}_{s}^{t+1}\,=\,\frac{1}{n}\sum_{i=1}^{n}X_{s}^{t+1,i},\qquad t+1\leq s\leq T,
\]

\[
\bar{X}_{t}^{t+1}\,=\,\frac{1}{n}\sum_{i=1}^{n}X_{t}^{t,i}.
\]

Time-consistency and Theorem 4.1 imply that there is $a^{0}$
$\in$ $\mathcal{M}$ such that

\[
\begin{array}{lll}
\Psi_{t,T}\left(\bar{X}^{t}\right) & = & \Psi_{t,T}\left(\left(\bar{X}^{t}\right)1_{\left\{ t\right\} }+\Psi_{t+1,T}\left(\bar{X}^{t}\right)1_{[t+1,\infty)}\right)\\
 & = & \frac{\prec\left(\bar{X}^{t}\right)1_{\left\{ t\right\} }+\Psi_{t+1,T}\left(\bar{X}^{t}\right)1_{[t+1,\infty)},a^{0}\succ_{t,T}}{\prec1,a^{0}\succ_{t,T}}.
\end{array}
\]

Letting $Y$ $:=$ $\left(\bar{X}^{t}\right)1_{\left\{ t\right\} }+\Psi_{t+1,T}\left(\bar{X}^{t}\right)1_{[t+1,\infty)},$
we have that

\[
\prec Y,a^{0}\succ\;=E\left(\bar{X}_{t}^{t}\triangle a_{t}^{0}+\Psi_{t+1,T}\left(\bar{X}^{t}\right)\left(\underset{j=t+1}{\overset{T}{\sum}}\triangle a_{j}^{0}\right)\mid\mathcal{F}_{t}\right).
\]

In addition,

\[
\begin{array}{lll}
Y_{t} & := & \bar{X}_{t}^{t}\triangle a_{t}^{0}+\Psi_{t+1,T}\left(\bar{X}^{t}\right)\left(\underset{j=t+1}{\overset{T}{\sum}}\triangle a_{j}^{0}\right)\\
 & \leq & \bar{X}_{t}^{t+1}\triangle a_{t}^{0}+\Psi_{t+1,T}\left(\bar{X}^{t+1}\right)\left(\underset{j=t+1}{\overset{T}{\sum}}\triangle a_{j}^{0}\right)\\
 & =: & Y_{t+1}.
\end{array}
\]

Let  $C$ $:=$ $\left\{ \underset{j=t+1}{\overset{T}{\sum}}\triangle a_{j}^{0}>0\right\} ,$ and note that
  $P$$\left(C\right)$
$=$ $1,$ since $a^{0}$ $\in$ $D_{0,T}^{e}.$   Therefore, $P(A\cap C)$ $=$ $P\left(A\right)$ $>$
$0$ and

\[
Y_{t}<Y_{t+1}\qquad in\quad A\cap C.
\]

Then, there exists an event $B$ $\in$ $\mathcal{F}_{t}$, with $P\left(B\right)>0$, such that

\[
E\left(Y_{t}\mid\mathcal{\mathcal{F}}_{t}\right)<E\left(Y_{t+1}\mid\mathcal{\mathcal{F}}_{t}\right)\qquad \text{in}\quad B.
\]

Finally, by time-consistency and last inequality,   the following display holds in $B$

\[
\begin{array}{lll}
\Psi_{t,T}\left(\bar{X}^{t}\right) & = & \frac{E\left(Y_{t}\mid\mathcal{F}_{t}\right)}{\prec1,a^{0}\succ_{t,T}}\\
 & < & \frac{E\left(Y_{t+1}\mid\mathcal{F}_{t}\right)}{\prec1,a^{0}\succ_{t,T}}\\
 & \leq & {\rm ess}\, \sup_{a\in\mathcal{M}}\frac{\prec\left(\bar{X}^{t+1}\right)1_{\left\{ t\right\} }+\Psi_{t+1,T}\left(\bar{X}^{t+1}\right)1_{[t+1,\infty)},a\succ_{t,T}}{\prec1,a\succ_{t,T}}\\
 & = & \Psi_{t,T}\left(\left(\bar{X}^{t+1}\right)1_{\left\{ t\right\} }+\Psi_{t+1,T}\left(\bar{X}^{t+1}\right)1_{[t+1,\infty)}\right)\\
 & = & \Psi_{t,T}\left(\bar{X}^{t+1}\right),
\end{array}
\]

which is a contradiction.
\end{proof}

Notice   that the proof of the previous theorem follows the same lines as that of  Theorem \ref{MainThm}. However, the key difference consist in  avoiding  one of the main  assumptions in this theorem, concerning the boundedness  from  below in the set $\mathcal{Q}_{s}$. This is achieved    noting that coherency and stability under
concatenation, together with Theorem \ref{Thm4.1},  allow us to write the monetary utility function process in a simple way. Thus, 
for each $X$ $\in$ $\mathcal{R}_{0,T}^{\infty},$
there is  $a$ $\in$
$\mathcal{M},$ such that
\[
\phi_{s,T}\left(X\right)=\frac{\prec X,a\succ_{s,T}}{\prec1,a\succ_{s,T}},\;\;\;\;s\in [0,T] \cap\mathbb{N}.
\]
 Hence, it is not necessary to have a global property such as  boundedness  from below in $\mathcal{Q}_{s}$.
 % like in Theorem \ref{MainThm}. 
 The result follows from the properties of the set  $D_{0,T}^{e}$ and the time consistency assumption, illustrating once again  the fact that the preservation of worst case portfolios of insurance versions is naturally linked with the property of time consistency.

We conclude this paper presenting a result that allow us to determine
when the risk of modifying two given processes is comparable, even
though the processes are not.

\begin{proposition}
Let $\phi$ be a concave monetary utility function that
is continuous for bounded decreasing sequences in $R_{t,T}^{\infty}$.
Assume that there exists a matrix $A$ $\in$ $Mat_{\left(T-t+1\right)\times\left(T-t+1\right)}\left(\mathcal{F}_{1}\right)$
with the following properties:

$\left(i\right)$ $\left(1,...,1\right)$ $\in$ $\mathbb{R}^{T-t+1}$
is an eigenvector of $A$ with eigenvalue $1.$

$\left(ii\right)$ $\left(A_{i,j}\right)$$\geq$ $0,$
for all $i$, $j$ $\in$ $\left\{ 1,...,T-t+1\right\} .$

$\left(iii\right)$ $X$ $\in$ $\mathcal{R}_{t,T}^{\infty}$
and $AX$ $\in$ $\mathcal{C}_{\phi}$ imply $X$ $\in$ $\mathcal{C}_{\phi}$

$\left(iv\right)$
\[
A\left(\sum_{i=1}^{n}\tilde{X}^{i}\right)\leq\sum_{i=1}^{n}\bar{X}^{i},
\]
for $\tilde{X}^{i}$, $\bar{X}^{i}$ $\in$ $\mathcal{R}_{t,T}^{\infty}$
, $i$ $=$ $1$,...,$n$.

Then,
\[
\Psi\left(A\left(\frac{1}{n}\sum_{i=1}^{n}\tilde{X}^{i}\right)\right)\leq\Psi\left(A\left(\frac{1}{n}\sum_{i=1}^{n}\bar{X}^{i}\right)\right),
\]
where $\Psi\left(X\right)  := -\phi\left(-X\right) $ for all $X$ $\in$ $\mathcal{R}_{t,T}^{\infty}$.
\end{proposition}

\begin{proof}
Define $\bar{\phi}$$:$ $\mathcal{R}_{t,T}^{\infty}$ $\longrightarrow$$L^{\infty}\left(\mathcal{F}_{1}\right)$
as

\[
\bar{\phi}\left(X\right)=\phi\left(AX\right),\quad X\in \mathcal{R}_{t,T}^{\infty}.
\]
It is not difficult to see that $\bar{\phi}$ is a concave monetary
utility function that is continuous for bounded decreasing sequences.
Let us denote by $\bar{\Psi}$  the functional given as $\bar{\Psi}\left(X\right) := -\bar{\phi}\left(-X\right)$ for all $X$ $\in$ $\mathcal{R}_{t,T}^{\infty}$. By Theorem 3.16 from \cite{cheridito2006dynamic}, we have that

\[
\Psi\left(X\right)={\rm ess}\, \sup_{a\in D_{t,T}}\left\{ \prec X,a\succ_{t,T}+\phi^{\#}\left(a\right)\right\} ,
\]

\[
\bar{\Psi}\left(X\right)={\rm ess}\, \sup_{a\in D_{t,T}}\left\{ \prec X,a\succ_{t,T}+\bar{\phi}^{\#}\left(a\right)\right\} .
\]

Since $\bar{\Psi}\left(X\right)$ $=$ $\Psi\left(AX\right),$
then

\[
\bar{\Psi}\left(X\right)={\rm ess}\, \sup_{a\in D_{t,T}}\left\{ \prec AX,a\succ_{t,T}+\phi^{\#}\left(a\right)\right\} .
\]
By hypothesis, it is clear that

\[
\phi^{\#}\left(a\right)={\rm ess}\, \inf_{X\in C_{\phi}}\prec X,a\succ_{t,T}\,\leq\, {\rm ess}\, \inf_{X\in C_{\bar{\phi}}}\prec X,a\succ_{t,T}\,=\bar{\phi}^{\#}\left(a\right).
\]
Therefore, for all $a$ $\in$ $D_{t,T}$ the following inequality
holds

\[
\prec A\left(\frac{1}{n}\sum_{i=1}^{n}\tilde{X}^{i}\right),a\succ_{t,T}+\phi^{\#}\left(a\right)\:\leq\:\prec\frac{1}{n}\sum_{i=1}^{n}X^{i},a\succ_{t,T}+\bar{\phi}^{\#}\left(a\right),
\]
and the conclusion  follows. 
\end{proof}
\bibliographystyle{abbrv}
\bibliographystyle{apalike}

\bibliographystyle{agsm}
\bibliography{HH-MM-20152303}

\begin{thebibliography}{10}

\bibitem{artzner1997a1}
P.~Artzner, F.~Delbaen, J.~M. Eber, and D.~Heath.
\newblock Thinking coherently.
\newblock {\em Risk}, 11:68--71, 1997.

\bibitem{artzner1999coherent}
P.~Artzner, F.~Delbaen, J.-M. Eber, and D.~Heath.
\newblock Coherent measures of risk.
\newblock {\em Mathematical Finance}, 9(3):203--228, 1999.

\bibitem{burgert2006consistent}
C.~Burgert and L.~R{\"u}schendorf.
\newblock Consistent risk measures for portfolio vectors.
\newblock {\em Insurance: Mathematics and Economics}, 38(2):289--297, 2006.

\bibitem{cerreia2011risk}
S.~Cerreia-Vioglio, F.~Maccheroni, M.~Marinacci, and L.~Montrucchio.
\newblock Risk measures: rationality and diversification.
\newblock {\em Mathematical Finance}, 21(4):743--774, 2011.

\bibitem{cheridito2006dynamic}
P.~Cheridito, F.~Delbaen, M.~Kupper, et~al.
\newblock Dynamic monetary risk measures for bounded discrete-time processes.
\newblock {\em Electronic Journal of Probability}, 11(3):57--106, 2006.

\bibitem{cheridito2011composition}
P.~Cheridito and M.~Kupper.
\newblock Composition of time-consistent dynamic monetary risk measures in
  discrete time.
\newblock {\em International Journal of Theoretical and Applied Finance},
  14(01):137--162, 2011.

\bibitem{danielsson2001academic}
J.~Danielsson, P.~Embrechts, C.~Goodhart, C.~Keating, F.~Muennich, O.~Renault,
  and H.~S. Shin.
\newblock An academic response to basel ii, 2001.

\bibitem{ekeland2012comonotonic}
I.~Ekeland, A.~Galichon, and M.~Henry.
\newblock Comonotonic measures of multivariate risks.
\newblock {\em Mathematical Finance}, 22(1):109--132, 2012.

\bibitem{embrechts2013model}
P.~Embrechts, G.~Puccetti, and L.~R{\"u}schendorf.
\newblock Model uncertainty and var aggregation.
\newblock {\em Journal of Banking \& Finance}, 37(8):2750--2764, 2013.

\bibitem{filipovic2012approaches}
D.~Filipovic, M.~Kupper, and N.~Vogelpoth.
\newblock Approaches to conditional risk.
\newblock {\em SIAM Journal on Financial Mathematics}, 3(1):402--432, 2012.

\bibitem{follmer2011stochastic}
H.~F{\"o}llmer and A.~Schied.
\newblock {\em Stochastic finance: an introduction in discrete time}.
\newblock Walter de Gruyter, 2011.

\bibitem{rachev1998mass}
S.~T. Rachev and L.~R{\"u}schendorf.
\newblock {\em Mass Transportation Problems: Volume I: Theory}, volume~1.
\newblock Springer Science \& Business Media, 1998.

\bibitem{ruschendorf2006law}
L.~R{\"u}schendorf.
\newblock Law invariant convex risk measures for portfolio vectors.
\newblock {\em Statistics \& Decisions}, 24(1/2006):97--108, 2006.

\bibitem{ruschendorf2012worst}
L.~R{\"u}schendorf.
\newblock Worst case portfolio vectors and diversification effects.
\newblock {\em Finance and Stochastics}, 16(1):155--175, 2012.

\bibitem{ruschendorf1990characterization}
L.~R{\"u}schendorf and S.~Rachev.
\newblock A characterization of random variables with minimum l 2-distance.
\newblock {\em Journal of Multivariate Analysis}, 32(1):48--54, 1990.

\bibitem{ruschendorf2002n}
L.~R{\"u}schendorf and L.~Uckelmann.
\newblock On the n-coupling problem.
\newblock {\em Journal of multivariate analysis}, 81(2):242--258, 2002.

\end{thebibliography}

\end{document}